\documentclass[reqno,11pt,a4paper]{amsart}
\usepackage[margin=2.5cm]{geometry}
\usepackage{newpxtext}
\usepackage{newtxmath}
\usepackage[all]{xy}
\usepackage[usenames,dvipsnames]{xcolor}
\usepackage{colortbl}
\usepackage{tikz}
\usetikzlibrary{positioning}
\usepackage{tikz-3dplot}
\usepackage{enumitem}
\usepackage{cite}
\usepackage{hyperref}
\usepackage{mathtools}
\usepackage[up]{caption}
\usepackage[labelformat=simple]{subcaption}
\newcommand{\orcid}[1]{\,\resizebox{8px}{!}{\href{https://orcid.org/#1}{\includegraphics{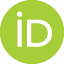}}}}
\graphicspath{{images/}}
\hypersetup{
    colorlinks,
    linkcolor={red!50!black},
    citecolor={blue!50!black},
    urlcolor={blue!80!black}
}

\AtBeginDocument{%
\def\MR#1{}
}
\theoremstyle{plain}
\newtheorem{thm}{Theorem}
\newtheorem{pro}[thm]{Proposition}
\newtheorem{conj}[thm]{Conjecture}
\theoremstyle{definition}
\newtheorem{dfn}[thm]{Definition}
\newtheorem{rem}[thm]{Remark}
\newtheorem{exa}[thm]{Example}
\DeclareMathOperator{\rk}{rk}
\DeclareMathOperator{\Hom}{Hom}
\DeclareMathOperator{\Pic}{Pic}
\DeclareMathOperator{\Newt}{Newt}
\DeclareMathOperator{\Aut}{Aut}
\DeclareMathOperator{\GL}{GL}
\newcommand{\sslash}{\mathbin{/\mkern-6mu/}}
\newcommand{\oo}{\mathcal{O}}
\newcommand{\QQ}{\mathbb{Q}}
\newcommand{\PP}{\mathbb{P}}
\newcommand{\ZZ}{\mathbb{Z}}
\newcommand{\FF}{\mathbb{F}}
\newcommand{\CC}{\mathbb{C}}
\newcommand{\LL}{\mathbb{L}}
\newcommand{\Cstar}{\CC^\times}
\newcommand{\cB}{\mathcal{B}}
\newcommand{\cV}{\mathcal{V}}
\newcommand{\cZ}{\mathcal{Z}}
\newcommand{\Ghat}{\widehat{G}}
\newcommand{\QQFano}{$\QQ$\nobreakdash-Fano}
\begin{document}
\author[T. Coates]{Tom Coates\orcid{0000-0003-0779-9735}}
\address{Department of Mathematics\\Imperial College London\\180 Queen's Gate\\London\\UK}
\email{t.coates@imperial.ac.uk}
\author[L.\,Heuberger]{Liana Heuberger\orcid{0000-0002-8038-9620}}
\address{Department of Mathematical Sciences\\University of Bath\\Claverton Down\\Bath\\UK}
\email{lh2457@bath.ac.uk}
\author[A.\,M.\,Kasprzyk]{Alexander M.~Kasprzyk\orcid{0000-0003-2340-5257}}
\address{School of Mathematical Sciences\\University of Nottingham\\Nottingham\\UK}
\email{a.m.kasprzyk@nottingham.ac.uk}
\title{Mirror Symmetry, Laurent Inversion and the Classification of $\QQ$-Fano Threefolds}
\begin{abstract}
We describe recent progress in a program to understand the classification of three-dimensional Fano varieties with $\mathbb{Q}$-factorial terminal singularities using mirror symmetry. As part of this we give an improved and more conceptual understanding of Laurent inversion, a technique that sometimes allows one to construct a Fano variety $X$ directly from a Laurent polynomial $f$ that corresponds to it under mirror symmetry.
\end{abstract}
\maketitle
\section{Introduction}\label{sec:introduction}
\QQFano{} threefolds are three-dimensional Fano varieties with at worst~$\QQ$-factorial terminal singularities. They play an important role in the Minimal Model Program~\cite{BirkarCasciniHaconMcKernan2010,HaconMcKernan2010,Mori1982,Mori1988,Kollar1989}. In this paper we consider the classification of \QQFano{} threefolds up to~$\QQ$\nobreakdash-Gorenstein (qG) deformation. It is known that there are finitely many deformation families~\cite{KollarMiyaokaMoriTakagi2000}, and many such families have been constructed explicitly~\cite{CampanaFlenner1993,Sano1995,Sano1996,IanoFletcher2000,Takagi2002,Kasprzyk2006,BrownKerberReid2012,ProkhorovReid2016,BrownKasprzykQureshi2018,Ducat2018,CoughlanDucat2020}, but the classification is still far from understood. We will describe a new approach to the classification problem, which is motivated by mirror symmetry. This approach has been successful in recovering the (known) classifications of smooth Fano varieties in dimensions two and three~\cite{CoatesCortiGalkinGolyshevKasprzyk2013,CoatesCortiGalkinKasprzyk2016}, and in classifying singular del~Pezzo surfaces~\cite{AkhtarCoatesCortiHeubergerKasprzykOnetoPetracciPrinceTveiten2016,KasprzykNillPrince2017,CortiHeuberger2017}. The key idea is that there should be a one-to-one correspondence between equivalence classes of Fano varieties with mild singularities and equivalence classes of certain Laurent polynomials. The equivalence relation on Fano varieties here is qG\nobreakdash-deformation; the equivalence relation on Laurent polynomials is mutation~\cite{AkhtarCoatesGalkinKasprzyk2012}. 

\begin{dfn}
  A Fano variety is of~\emph{class TG} (for `toric generisation') if it occurs as the general fiber of a qG\nobreakdash-degeneration with reduced fibers and special fiber a normal toric variety.
\end{dfn}

\noindent The vast majority of \QQFano{} varieties are expected to be of class TG; cf.~\cite{PostinghelUrbinati2016}.

\begin{conj} \label{conj:mirror}
    There is a bijective correspondence between qG\nobreakdash-deformation families of \QQFano{} threefolds~$X$ of class TG and mutation-equivalence classes of rigid maximally mutable Laurent polynomials~$f$ in three variables. Under this correspondence the regularized quantum period~$\Ghat_X$ coincides with the classical period~$\pi_f$; see~\S\ref{sec:mirror_correspondence} for more on this.
\end{conj}

\noindent
Conjecture~\ref{conj:mirror} is a specialisation of~\cite[Conjecture~5.1]{CoatesKasprzykPittonTveiten2021} to our three-dimensional setting.
As we explain in~\S\ref{sec:mirror_correspondence}, when~$X$ and~$f$ correspond under Conjecture~\ref{conj:mirror} we expect that there is a qG\nobreakdash-degeneration from ~$X$ to the toric variety~$X_f$ defined by the spanning fan of the Newton polytope of~$f$.

Establishing Conjecture~\ref{conj:mirror} will require substantial advances in the Gross--Siebert program~\cite{GrossSiebert2003,GrossSiebert2006,GrossSiebert2010,GrossSiebert2013,ArguzGross2020,GrossHackingSiebert2016}, or in deformation theory (cf.~\cite{CortiFilipPetracci2020,CortiHackingPetracciInPreparation}). But nonethless, whilst the foundations of mirror symmetry are being developed, we can use Conjecture~\ref{conj:mirror} to fill in large parts of the classification of \QQFano{} threefolds that were previously unknown. That is, we can use methods inspired by Conjecture~\ref{conj:mirror}, and by mirror symmetry more broadly, to~\emph{construct} large parts of the classification. As well as giving many new families of \QQFano{} threefolds, these constructions also give supporting evidence for the rich conjectural picture predicted by mirror symmetry. 

Our approach is as follows. We first look for rigid maximally mutable Laurent polynomials (MMLPs)~\cite{CoatesKasprzykPittonTveiten2021}, by specifying a class of lattice polytopes and then searching algorithmically for all rigid MMLPs~$f$ such that the Newton polytope~$\Newt{f}$ lies in this class. Initially here we insist that~$\Newt{f}$ is a three-dimensional lattice polytope with one lattice point in the strict interior; such lattice polytopes are called~\emph{canonical} and have been classified~\cite{Kasprzyk2010,Zenodo-canonical3}. We then expand the search to include certain Laurent polynomials~$f$ such that~$\Newt{f}$ has two lattice points in the strict interior~\cite{BallettiKasprzyk2016}. In this way we obtain a large collection of rigid MMLPs in three variables. We partition this set of Laurent polynomials into mutation-equivalence classes and then, for each class, attempt to construct a deformation family of \QQFano{} threefolds that realises this class via Conjecture~\ref{conj:mirror}.
The method -- Laurent inversion~\cite{CoatesKasprzykPrince2019} -- that we use to construct~$X$ from~$f$ is also inspired by mirror symmetry: see~\S\ref{sec:towers} below for a detailed discussion, and~\S\ref{sec:laurent_inversion} for several examples. 

\subsection{The Mirror Correspondence}\label{sec:mirror_correspondence}
If a \QQFano{} threefold~$X$ corresponds, via Conjecture~\ref{conj:mirror}, to a Laurent polynomial~$f$ then the regularized quantum period of~$X$ matches the classical period of~$f$~\cite{CoatesCortiGalkinGolyshevKasprzyk2013}. Here the regularized quantum period~$\Ghat_X$ of~$X$ is a generating function
\[
    \Ghat_X(t) = 1 + \sum_{d=2}^\infty c_d t^d
\]
where~$c_d = r_d d!$ and~$r_d$ is a certain genus-zero Gromov--Witten invariant of~$X$, and the classical period~$\pi_f$ of~$f \in \CC[x_1^{\pm 1},\ldots,x_n^{\pm 1}]$ is
\[
    \pi_f(t) = \frac{1}{(2 \pi i)^n}
    \int_{(S^1)^n} \frac{1}{1 - t f} \frac{dx_1}{x_1} \cdots \frac{dx_n}{x_n}
\]
which expands as a power series
\[
    \pi_f(t) = \sum_{d=0}^\infty \kappa_d t^d
\]
with~$\kappa_d$ the coefficient of the constant monomial in~$f^d$. Gromov--Witten invariants are deformation invariant, so the regularized quantum period~$\Ghat_X$ is a qG\nobreakdash-deformation invariant of~$X$. On the other side of the correspondence, the classical period~$\pi_f$ is invariant under mutation of~$f$.

If the Fano variety~$X$ corresponds to the Laurent polynomial~$f$ via Conjecture~\ref{conj:mirror} then it is expected that there is a qG\nobreakdash-degeneration with general fiber~$X$ and special fiber~$X_f$, where~$X_f$ is the toric variety defined by the spanning fan of the Newton polytope~$\Newt{f}$. Thus one can hope to recover~$X$ from the Laurent polynomial~$f$ by smoothing the toric variety~$X_f$, which is in general highly singular. The coefficients of~$f$ should therefore somehow encode a logarithmic structure~\cite{Kato1989} on the central fiber~$X_f$ of this degeneration. From this perspective one can think of Laurent inversion -- the technique that we use to construct \QQFano{} threefolds -- as attempting to construct the expected smoothing of~$X_f$ as an embedded deformation of~$X_f$ inside an ambient toric variety built from~$f$. For more on this, see the work of Doran--Harder~\cite{DoranHarder2016} and also~\cite[\S8]{CoatesKasprzykPrince2019}. If the conjectural picture described above, with the \QQFano{} variety~$X$ degenerating to a toric variety~$X_f$, is correct, then one can think of~$X$ as corresponding to a general point on an appropriate component of the Hilbert scheme, and~$X_f$ as giving a point on the boundary of that component. From this point of view, the discussion in~\S\ref{sec:two_rigid_MMLPs} is particularly instructive. We exhibit two rigid MMLPs~$f_1$ and~$f_2$ with the same Newton polytope~$P$ -- so that~$X_{f_1} = X_{f_2} = X_P$ -- but different classical periods~$\pi_{f_1}$,~$\pi_{f_2}$. We then use Laurent inversion to construct \QQFano{} threefolds~$X_1$ and~$X_2$ which correspond respectively to~$f_1$ and~$f_2$ under Conjecture~\ref{conj:mirror}. Since the classical periods~$\pi_{f_1}$ and~$\pi_{f_2}$ are different, we have that~$\Ghat_{X_1} \ne \Ghat_{X_2}$; thus~$X_1$ and~$X_2$ are not isomorphic, or even qG\nobreakdash-deformation equivalent. This means that~$X_1$ and~$X_2$ lie on different components of the Hilbert scheme, and the singular toric variety~$X_P$ lies in the intersection of these components. The choice of rigid MMLP with Newton polytope~$P$ -- that is, the choice of~$f_1$ or~$f_2$ -- corresponds to choosing a component of the Hilbert scheme that contains~$X_P$ and gives a smoothing of~$X_P$.

\subsection{The Graded Ring Database}
Miles Reid and his collaborators have pioneered the study of \QQFano{} threefolds using graded ring methods~\cite{AltinokBrownReid2002,IanoFletcher2000,BrownKasprzyk2022,Altinok98,BrownSuzuki2007a,BrownSuzuki2007b}. They have determined a set of~$39\,550$ rational functions that contains all Hilbert series of \QQFano{} threefolds that satisfy a semistability condition\footnote{See~\cite{BrownKasprzyk2022} for a precise discussion.} and have Picard rank~1. In practice all known \QQFano{} threefolds have Hilbert series contained in this dataset, regardless of semistability or Picard rank, so we will ignore these conditions in what follows. The Hilbert series of a Fano variety~$X$ is the generating series for the dimensions of the graded pieces of the anticanonical ring
\begin{equation}\label{eq:anticanonical}
R(X, {-K}_X) = \bigoplus_{n = 0}^\infty H^0(X, {-n} K_X).
\end{equation}
Note that the Hilbert series is invariant under qG\nobreakdash-deformation of~$X$. Choosing a minimal set of homogeneous generators for the anticanonical ring~\eqref{eq:anticanonical} determines an embedding of~$X$ into weighted projective space~$w\PP$, and one can use the Hilbert series to estimate the weights and codimension of such an embedding: see~\cite[\S3]{BrownKasprzyk2022}. The Hilbert series also determines the~\emph{genus}~$g(X) \coloneqq h^0(X,{-K}_X) - 2$. When~$X$ is smooth,~$g(X)$ is the genus of the curve given by intersecting~$X$ with two generic hyperplanes in~$w\PP$.

The dataset of possible Hilbert series is recorded in the Graded Ring Database~\cite{GRDB-fano3,Zenodo-fano3}. One can think of this data as giving a numerical sketch of the possible `geography' of \QQFano{} threefolds. A point to note is that the combinatorial methods used to produce the dataset of possible Hilbert series do not guarantee the existence (or uniqueness) of a deformation family of \QQFano{} threefolds with that Hilbert series: there can be zero, one, or many such families. From this point of view, our work gives a new way to approach the realisation problem for a given possible Hilbert series~$H(t)$. If there is a \QQFano{} threefold~$X$ with Hilbert series~$H(t)$, and~$X$ corresponds under Conjecture~\ref{conj:mirror} to a rigid MMLP~$f$, then as discussed we expect that there is a qG\nobreakdash-degeneration from~$X$ to the toric variety~$X_f$. This toric variety is defined by the spanning fan of the polytope~$P = \Newt{f}$, and the Hilbert series of~$X_f$ is equal to the Ehrhart series of the dual polytope~$P^*$. One can therefore approach the realisation problem as follows.

Given a possible Hilbert series~$H(t)$ one can search for Fano polytopes~$P$ such that the Ehrhart series of~$P^*$ is equal to~$H(t)$. For each such~$P$ one can search for rigid MMLPs~$f$ with Newton polytope~$P$. Partitioning these Laurent polynomials into mutation-equivalence classes predicts the number of deformation families, as well as specific qG toric degenerations~$X_f$ of these families. One can then use Laurent inversion (as in~\S\ref{sec:laurent_inversion}), or search for toric complete intersection models (as in~\S\ref{sec:random_ci}), or use more traditional methods such as unprojection~\cite{BrownKerberReid2012,PapadakisReid2004} to construct each family. 

\subsection{The landscape of \QQFano{} threefolds}
\begin{figure}[tbp]
    \centering
    \begin{subfigure}{.48\textwidth}
      \includegraphics[width=\linewidth]{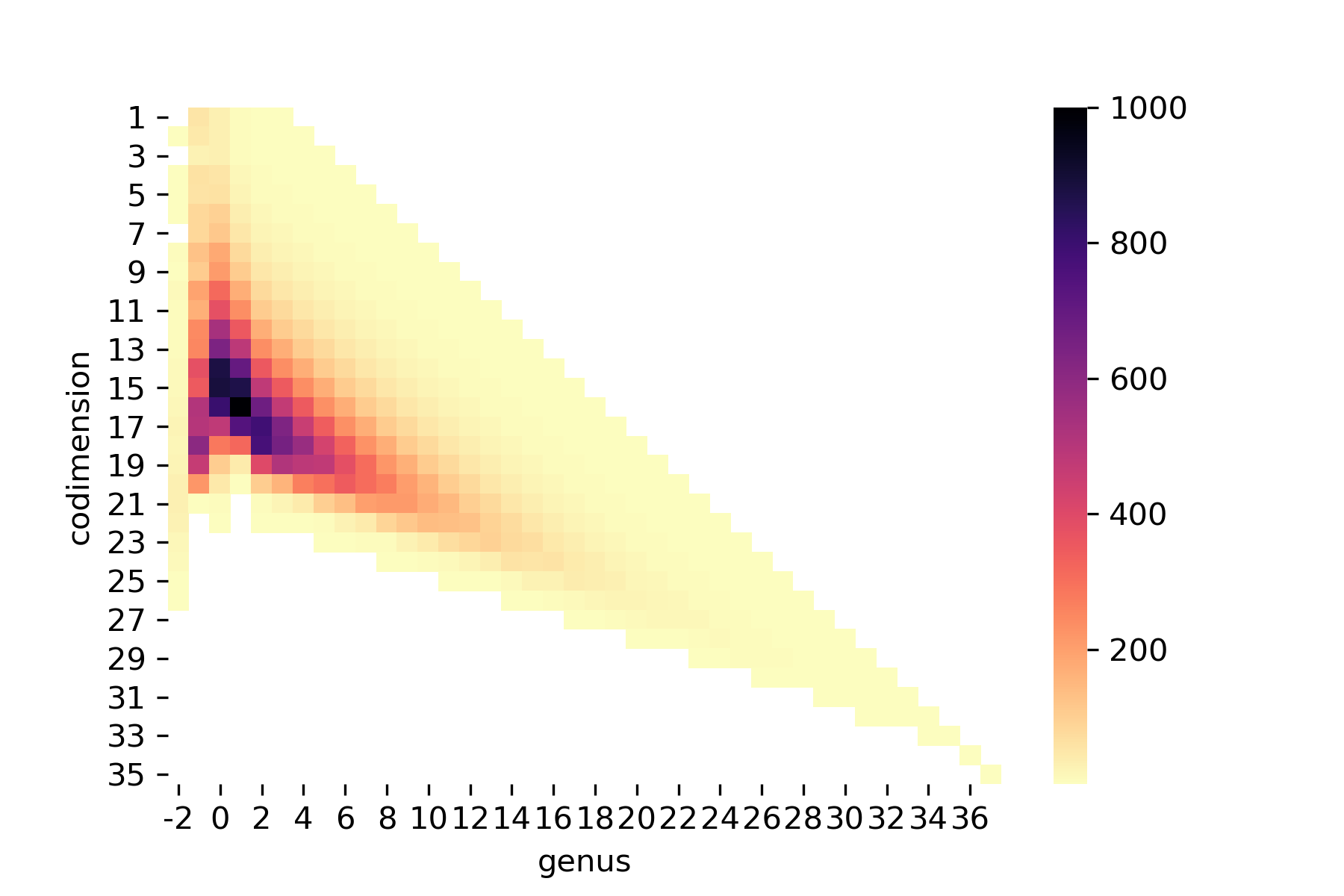}
      \caption{Potential~$\QQ$-Fano threefolds}
      \label{fig:grdb_heatmaps_a}
    \end{subfigure} \\
    \begin{subfigure}{.48\textwidth}
      \includegraphics[width=\linewidth]{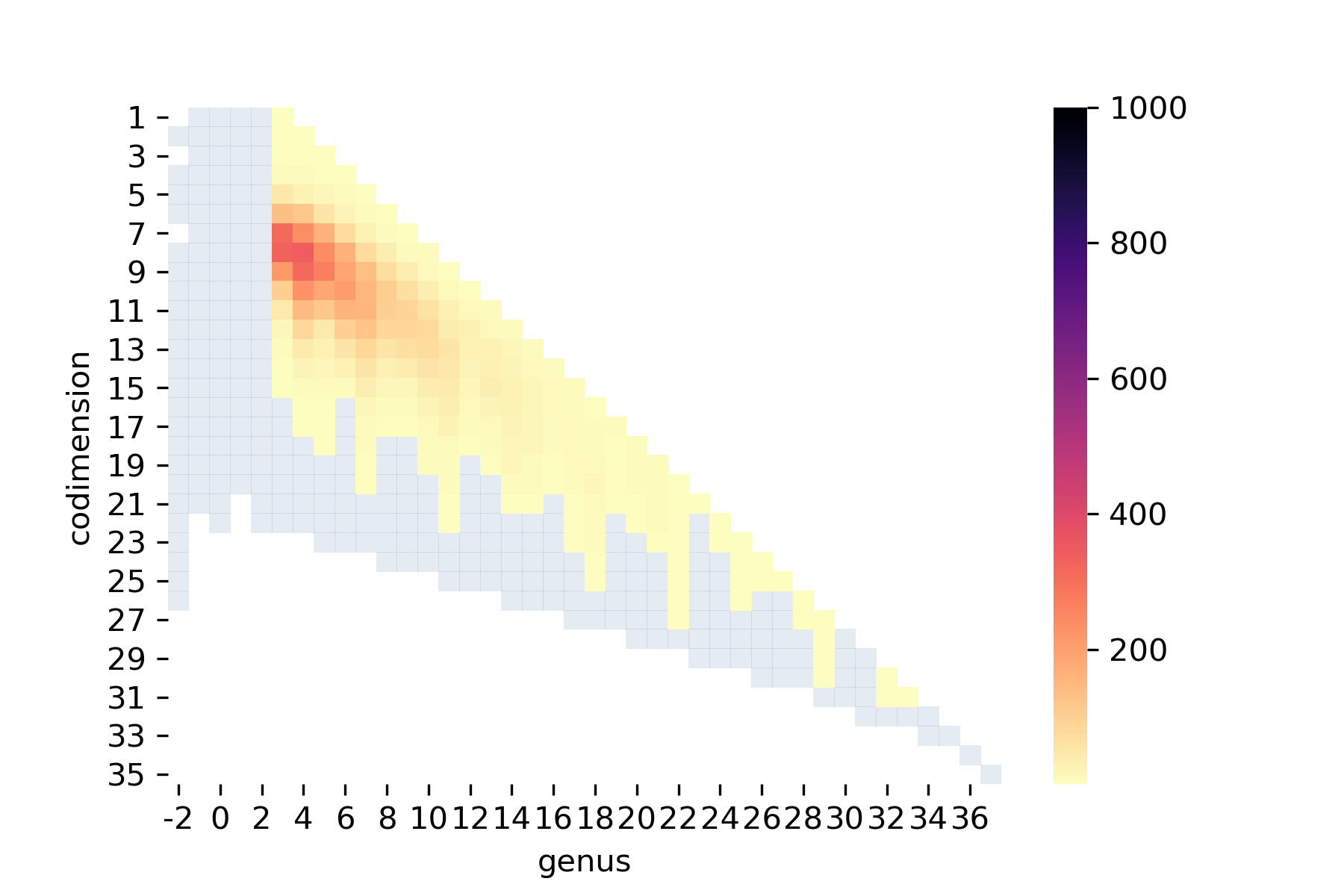}
      \caption{Rigid MMLPs on canonical polytopes}
      \label{fig:grdb_heatmaps_b}
    \end{subfigure} 
    \hfill
    \begin{subfigure}{.48\textwidth}
      \includegraphics[width=\linewidth]{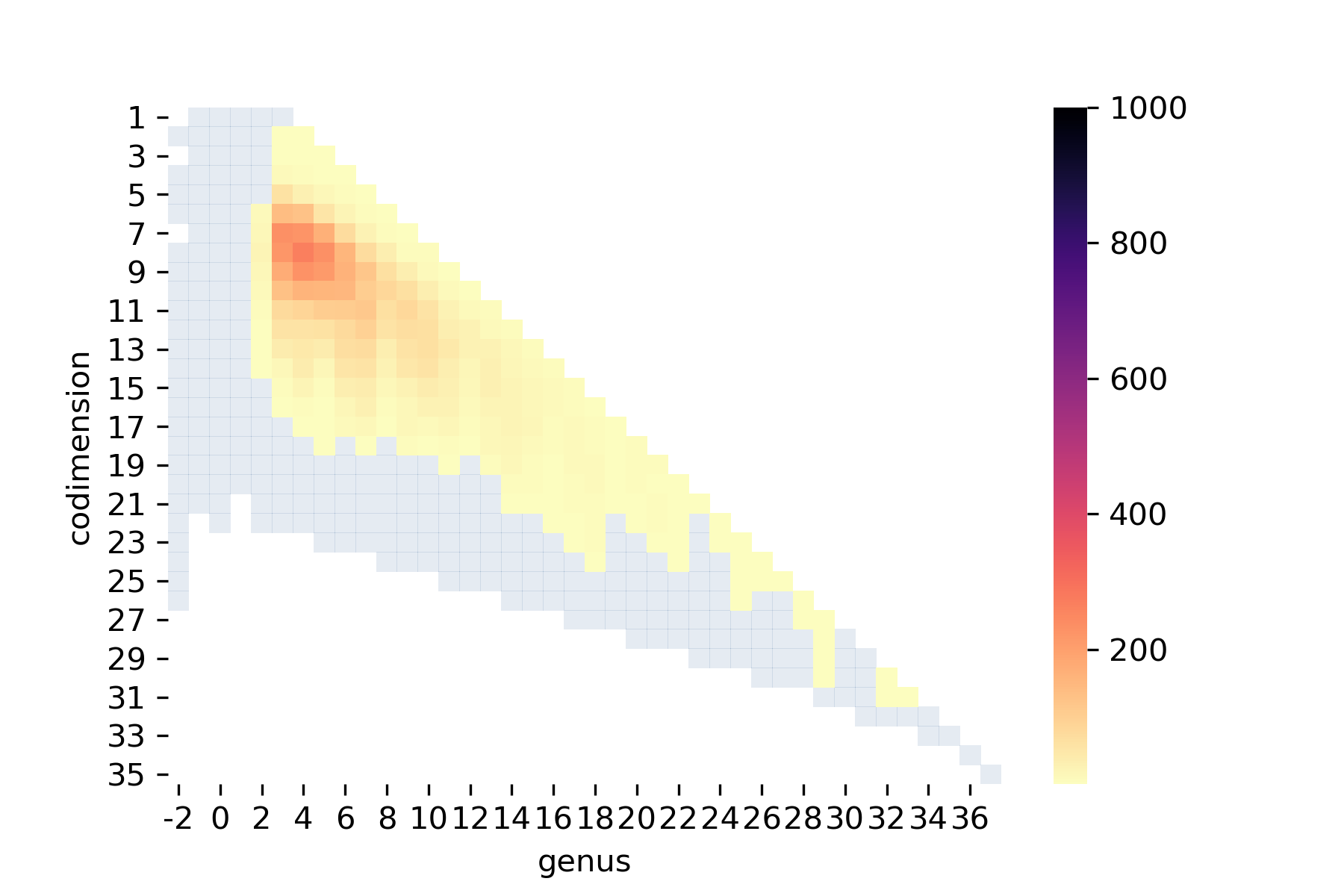}
      \caption{Rigid MMLPs on a sample of~$2$-point polytopes} 
      \label{fig:grdb_heatmaps_c}
    \end{subfigure} \\
    \begin{subfigure}{.48\textwidth}
      \includegraphics[width=\linewidth]{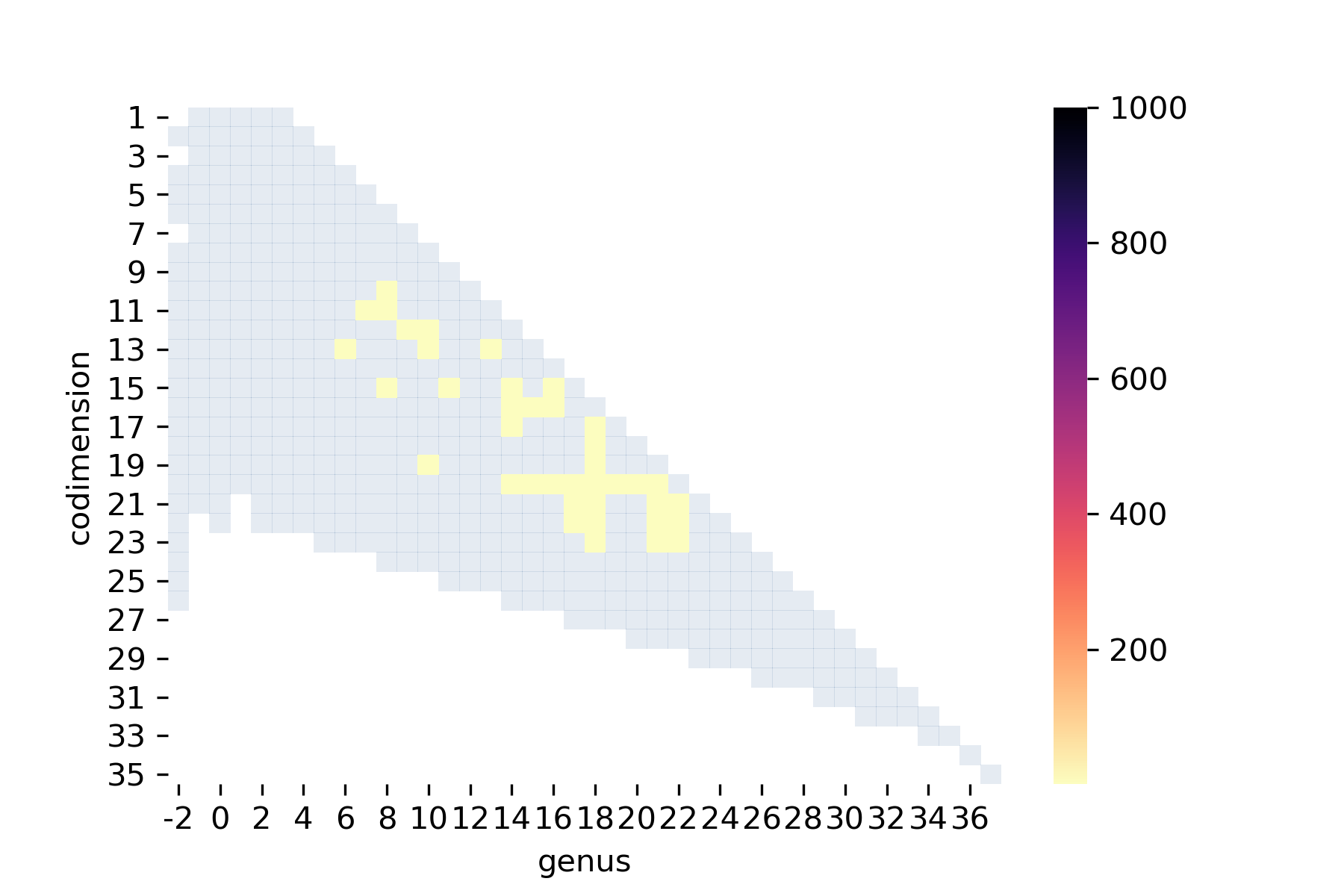}
      \caption{Examples constructed using Laurent inversion}
      \label{fig:grdb_heatmaps_d}
    \end{subfigure} 
    \hfill
    \begin{subfigure}{.48\textwidth}
      \includegraphics[width=\linewidth]{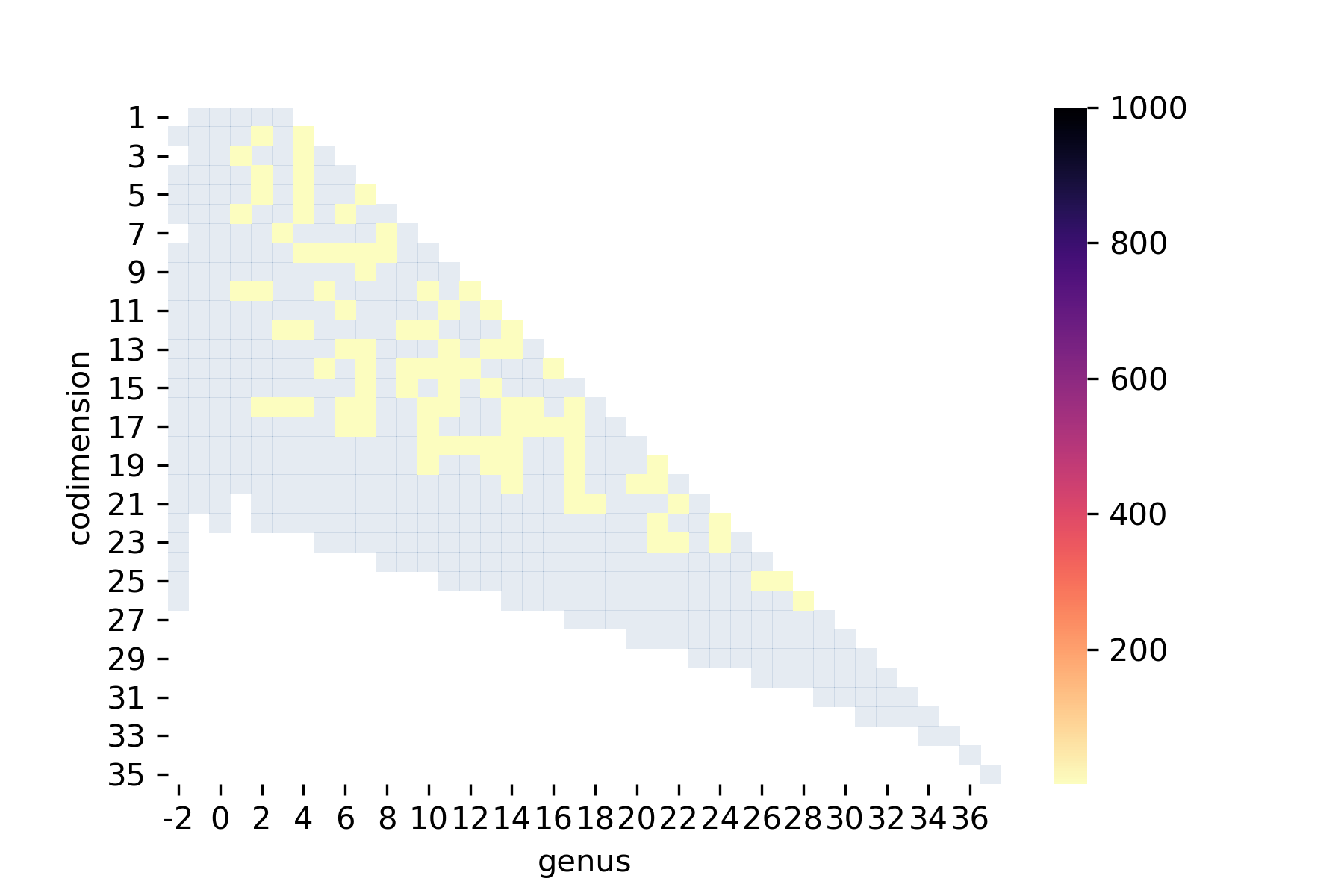}
      \caption{Randomly generated toric hypersurfaces}
      \label{fig:grdb_heatmaps_e}
    \end{subfigure} 
    \caption{ The distribution of potential Hilbert series for {\protect \QQFano} threefolds with Fano index~1: \subref{fig:grdb_heatmaps_a} from the Graded Ring Database; \subref{fig:grdb_heatmaps_b} from mutation-equivalence classes of rigid MMLPs with three-dimensional canonical Newton polytope; \subref{fig:grdb_heatmaps_c} from mutation-equivalence classes of rigid MMLPs with Newton polytope in a random sample of three-dimensional polytopes with two interior points; \subref{fig:grdb_heatmaps_d} from toric complete intersections constructed using Laurent inversion; \subref{fig:grdb_heatmaps_e} from randomly-generated quasismooth hypersurfaces in toric orbifolds of Picard rank~2. Hilbert series are recorded as pairs~$(c, g)$ where~$c$ is the estimated codimension and~$g = g(X)$ is the genus. Plots \subref{fig:grdb_heatmaps_b}--\subref{fig:grdb_heatmaps_e} have the shadow of plot \subref{fig:grdb_heatmaps_a} as background.}
    \label{fig:grdb_heatmaps}
\end{figure}

Figure~\ref{fig:grdb_heatmaps} gives three different views of the distribution of \QQFano{} threefolds:
\begin{itemize}
  \item Figure~\ref{fig:grdb_heatmaps_a} shows the landscape of possible Hilbert series for \QQFano{} threefolds as determined by the Graded Ring Database. Every Hilbert series of a \QQFano{} threefold is recorded here, but this analysis is purely numerical and ignores the realisation problem: each Hilbert series here may be represented by zero, one, or many \QQFano{} threefolds. 
  \item Figure~\ref{fig:grdb_heatmaps_b} and Figure~\ref{fig:grdb_heatmaps_c} give two views of the landscape assuming the conjectural correspondence between \QQFano{} threefolds and mutation-equivalence classes of rigid MMLPs. These are experimental and (necessarily) incomplete. Figure~\ref{fig:grdb_heatmaps_b} records the distribution from what we believe to be an almost-complete collection of rigid MMLPs~$f$ in three variables such that~$\Newt{f}$ is a canonical polytope~\cite{Zenodo-certain-MMLP}. Figure~\ref{fig:grdb_heatmaps_c} records the distribution from a random sample of rigid MMLPs in three variables such that~$\Newt{f}$ has two interior points.
  \item Figure~\ref{fig:grdb_heatmaps_d} and Figure~\ref{fig:grdb_heatmaps_e} give two views of the landscape based on genuine \QQFano{} threefolds. Figure~\ref{fig:grdb_heatmaps_d} records the distribution from toric complete intersections constructed from rigid MMLPs using Laurent inversion: see~\S\ref{sec:laurent_inversion} and~\cite{Heuberger2022}. Figure~\ref{fig:grdb_heatmaps_e} records the distribution from a random sample of quasismooth hypersurfaces in toric orbifolds of Picard rank~$2$: see~\S\ref{sec:random_ci}.
\end{itemize}
Comparing Figure~\ref{fig:grdb_heatmaps_b} and Figure~\ref{fig:grdb_heatmaps_c} with Figure~\ref{fig:grdb_heatmaps_a} predicts that there are many \QQFano{} threefolds with the same Hilbert series, particularly in fairly low codimension and low genus. Comparing Figure~\ref{fig:grdb_heatmaps_c} and Figure~\ref{fig:grdb_heatmaps_e} with Figure~\ref{fig:grdb_heatmaps_b} indicates how restricting the Newton polytope of our rigid MMLPs~$f$ to be canonical prevents us from realising parts of the possible \QQFano{} landscape: for example it forces~$g(X) \geq 3$.
\section{Maximally Mutable Laurent Polynomials}
In this section we define mutations and mutability, and give a criterion (Proposition~\ref{pro:rigid}) for a Laurent polynomial to be a rigid MMLP. We then describe our systematic search for rigid MMLPs~$f$ in three variables such that the Newton polytope of~$f$ is canonical. 

\subsection{Mutations}
Let~$N$ be a lattice,~$M = \Hom(N, \ZZ)$ be the dual lattice, and consider Laurent polynomials~$f \in \CC[N]$. A~\emph{mutation} is an automorphism
\begin{align*}
    \mu_{w,h} \colon \CC(N) & \to \CC(N) \\
    x^\gamma & \mapsto h^{\left \langle w, \gamma \right \rangle} x^{\gamma}
\end{align*}
defined by a primitive lattice vector~$w \in M$ called the~\emph{weight} and a Laurent polynomial~$h \in \CC[w^\perp \cap N]$ called the~\emph{factor}~\cite{AkhtarCoatesGalkinKasprzyk2012}. Here we can think of~$w$ as defining a~$\ZZ$-grading on~$\CC[N]$, with~$h$ lying in the degree-zero piece of that grading. In general, given a Laurent polynomial~$f \in \CC[N]$, the mutation~$g \coloneqq \mu_{w,h}(f)$ will be a rational function. If~$g \in \CC[N]$ is also a Laurent polynomial then we say that~$f$ is~\emph{mutable} with respect to~$(w,h)$.

\begin{exa}
    Mutability of a Laurent polynomial~$f$ imposes constraints on its coefficients. For example, consider the Laurent polynomial
    \[
        f_a = y + \frac{1}{xy} + \frac{a}{y} + \frac{x}{y} 
    \]
    in variables~$x$ and~$y$, where~$a$ is a parameter, and the mutation~$\mu_{w,h} \colon \CC(N) \to \CC(N)$ where~$N = \ZZ^2$,~$w = (0,1)$ and~$h = 1+x$. Then~$\mu_{w,h}$ sends~$x \mapsto x$,~$y \mapsto (1+x) y$, and~$f_a$ is mutable with respect to~$(w,h)$ if and only if~$a = 2$. To see this, write
    \[
        f_a = y + \frac{1 + a x + x^2}{xy}    
    \]
    Then~$f_a$ is mutable with respect to~$(w,h)$ if and only if~$1 + ax + x^2$ is divisible by~$1+x$, that is, if and only if~$a = 2$.
\end{exa}

\begin{exa}
    If the factor~$h$ is a monomial, then the mutation~$\mu_{w,h}$ is a monomial change of variables and every Laurent polynomial~$f \in \CC[N]$ is mutable with respect to~$(w,h)$, for any weight~$w$. We regard such mutations as trivial: see~\cite[Definition~2.1]{CoatesKasprzykPittonTveiten2021}.
\end{exa}

\begin{exa} \label{ex:divisibility}
    Suppose that~$\mu_{w,h}$ is a non-trivial mutation -- i.e.~$h$ is not a monomial -- and that~$f \in\CC[N]$ is mutable with respect to~$(w,h)$. We may choose an identification of~$N$ with~$\ZZ^n$ such that the weight~$w = (0,\ldots,0,1)$ is the~$n$th standard basis vector for the dual lattice~$M = (\ZZ^n)^\vee$. Write~$\CC[N] = \CC[x_1^{\pm 1},\ldots,x_{n-1}^{\pm 1},y^{\pm 1}]$, so that~$h$ is a Laurent polynomial in the variables~$x_1,\ldots,x_{n-1}$ and 
    \begin{equation} \label{eq:slice_f}
        f = \sum_{k=-a}^{b} f_k y^k
    \end{equation}
    for some positive integers~$a$,~$b$ and some Laurent polynomials~$f_k$ in the variables~$x_1,\ldots,x_{n-1}$. The mutation~$\mu_{w,h} f$ is
    \[
        \mu_{w,h} f = \sum_{k=-a}^{b} h^k f_k y^k
    \]
    and therefore~$f$ is mutable if and only if~$h^k$ divides~$f_{-k}$ for all~$k > 0$. Since for any Laurent polynomials~$g_1$,~$g_2$ we have
    \[
        \Newt\big(g_1 g_2\big) = \Newt(g_1) + \Newt(g_2)
    \]
    where the operation on the right-hand side is Minkowski sum of polytopes
    \[
        P_1 + P_2 = \Big\{ p_1 + p_2 : p_1 \in P_1, p_2 \in P_2 \Big \}
    \]
    it follows that~$f$ is mutable with respect to~$(w,h)$ only if the~$k$\nobreakdash-fold dilate of~$\Newt(h)$ is a Minkowski summand of~$\Newt(f_{-k})$ for~$1 \leq k \leq a$.
\end{exa}

\begin{dfn}
    Consider a Laurent polynomial~$f \in \CC[N]$ with Newton polytope~$P$:
    \[
        f = \sum_{p \in N \cap P} a_p x^p
    \]
    We will say that~$f$ is~\emph{normalised} if~$a_v = 1$ whenever~$v$ is a vertex of~$P$, and that~$f$ is~\emph{centered} if the origin lies in the strict interior of~$P$ and~$a_0 = 0$. 
\end{dfn}

\begin{exa}
    Consider the polytope~$P$ with ID~\href{http://grdb.co.uk/search/toricf3c?ID_cmp=in&ID=1523}{1523} in the GRDB database of three-dimensional canonical polytopes, as pictured in Figure~\ref{fig:P1523}.
    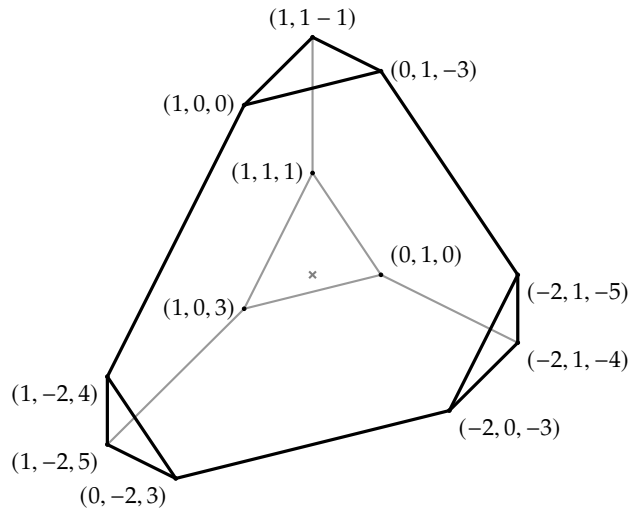
\begin{figure}[htb]
        \centering
        \begin{tikzpicture}[every node/.style={scale=0.8},scale=0.45]
        
        \draw[gray!80, thick] (2,8) -- (2,4);
        \draw[gray!80, thick] (4,1) -- (2,4);
        \draw[gray!80, thick] (0,0) -- (2,4);
        \draw[gray!80, thick] (0,0) -- (-4,-4);
        \draw[gray!80, thick] (4,1) -- (8,-1);
        \draw[gray!80, thick] (0,0) -- (4,1);

        \draw[black, very thick] (2,8) -- (0,6);
        \draw[black, very thick] (2,8) -- (4,7);
        \draw[black, very thick] (8,1) -- (4,7);
        \draw[black, very thick] (0,6) -- (-4,-2);
        \draw[black, very thick] (0,6) -- (4,7);
        \draw[black, very thick] (-4,-4) -- (-4,-2);
        \draw[black, very thick] (-4,-4) -- (-2,-5);
        \draw[black, very thick] (-4,-2) -- (-2,-5);
        \draw[black, very thick] (-2,-5) -- (6,-3);
        \draw[black, very thick] (6,-3) -- (8,-1);
        \draw[black, very thick] (8,1) -- (8,-1);
        \draw[black, very thick] (6,-3) -- (8,1);
        
        \filldraw[black] (-4,-4) circle (1.5pt) node[anchor=north east] {$(1,-2,5)$};
        \filldraw[black] (-4,-2) circle (1.5pt) node[anchor=north east] {$(1,-2,4)$};
        \filldraw[black] (-2,-5) circle (1.5pt) node[anchor=north east] {$(0,-2,3)$};
        
        \filldraw[black] (6,-3) circle (1.5pt) node[anchor=north west] {$(-2,0,-3)$};
        \filldraw[black] (8,-1) circle (1.5pt) node[anchor=north west] {$(-2,1,-4)$};
        \filldraw[black] (8,1) circle (1.5pt) node[anchor=north west] {$(-2,1,-5)$};
        
        \filldraw[black] (4,1) circle (1.5pt) node[anchor=south west] {$(0,1,0)$};
        \filldraw[black] (2,4) circle (1.5pt) node[anchor=east] {$(1,1,1)$};
        \filldraw[black] (0,0) circle (1.5pt) node[anchor=east] {$(1,0,3)$};
        
        \filldraw[black] (4,7) circle (1.5pt) node[anchor=west] {$(0,1,-3)$};
        \filldraw[black] (2,8) circle (1.5pt) node[anchor=south] {$(1,1-1)$};
        \filldraw[black] (0,6) circle (1.5pt) node[anchor=east] {$(1,0,0)$};
        
        \draw[gray, thick] (2.1,1.1) -- (2-0.1,1-0.1);
        \draw[gray, thick] (2+0.1,1-0.1) -- (2-0.1,1+0.1);
        \end{tikzpicture}
        \caption{The three-dimensional canonical polytope with GRDB ID~\href{http://grdb.co.uk/search/toricf3c?ID_cmp=in&ID=1523}{1523}.}
        \label{fig:P1523}
    \end{figure}
   $P$ has four triangular facets and four hexagonal facets. The automorphism group~$\Aut(P)$ is isomorphic to the symmetric group~$S_4$, and acts permuting the hexagonal facets. Let~$N$ be the lattice containing~$P$, let~$M = \Hom(N,\ZZ)$, and let~$w \in M$ be a supporting hyperplane for a hexagonal facet~$F$. The facet~$F$ is at height~$-1$ with respect to~$w$, so that if~$f$ is a Laurent polynomial with Newton polytope~$P$ then the expansion~\eqref{eq:slice_f} has~$a = 1$ (and~$b=2$).
    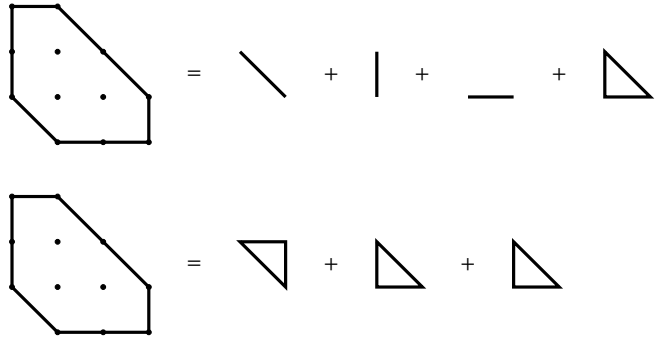
\begin{figure}[tb]
        \centering
        \begin{tikzpicture}[every node/.style={scale=0.8},scale=0.6]
        
        \draw[black, very thick] (1,0) -- (3,0);
        \draw[black, very thick] (3,0) -- (3,1);
        \draw[black, very thick] (3,1) -- (1,3);
        \draw[black, very thick] (1,3) -- (0,3);
        \draw[black, very thick] (0,3) -- (0,1);
        \draw[black, very thick] (1,0) -- (0,1);
        
        \draw[black, very thick] (5,2) -- (6,1);
        \draw[black, very thick] (8,1) -- (8,2);
        \draw[black, very thick] (10,1) -- (11,1);
        \draw[black, very thick] (13,1) -- (13,2) -- (14,1) -- cycle;
        
        \filldraw[black] (4,1.5) node{=};
        \filldraw[black] (7,1.5) node{+};
        \filldraw[black] (9,1.5) node{+};
        \filldraw[black] (12,1.5) node{+};

        \filldraw[black] (1,0) circle (1.5pt) node[anchor=north east]{};
        \filldraw[black] (3,0) circle (1.5pt) node[anchor=north east]{};
        \filldraw[black] (3,1) circle (1.5pt) node[anchor=north east]{};
        \filldraw[black] (1,3) circle (1.5pt) node[anchor=north east]{};
        \filldraw[black] (0,3) circle (1.5pt) node[anchor=north east]{};
        \filldraw[black] (0,1) circle (1.5pt) node[anchor=north east]{};
        
        \filldraw[black] (1,1) circle (1.5pt) node[anchor=north east]{};
        \filldraw[black] (1,2) circle (1.5pt) node[anchor=north east]{};
        \filldraw[black] (2,1) circle (1.5pt) node[anchor=north east]{};
        
        \filldraw[black] (2,2) circle (1.5pt) node[anchor=north east]{};
        \filldraw[black] (0,2) circle (1.5pt) node[anchor=north east]{};
        \filldraw[black] (2,0) circle (1.5pt) node[anchor=north east]{};
        
        \begin{scope}[shift={(0,-4.2)}]
            \draw[black, very thick] (1,0) -- (3,0);
            \draw[black, very thick] (3,0) -- (3,1);
            \draw[black, very thick] (3,1) -- (1,3);
            \draw[black, very thick] (1,3) -- (0,3);
            \draw[black, very thick] (0,3) -- (0,1);
            \draw[black, very thick] (1,0) -- (0,1);
            
            \draw[black, very thick] (5,2) -- (6,2)-- (6,1) -- cycle;
            \draw[black, very thick] (8,1) -- (9,1)-- (8,2) -- cycle;
            \draw[black, very thick] (11,1) -- (11,2)-- (12,1) -- cycle;

            \filldraw[black] (4,1.5) node{=};
            \filldraw[black] (7,1.5) node{+};
            \filldraw[black] (10,1.5) node{+};

            \filldraw[black] (1,0) circle (1.5pt) node[anchor=north east]{};
            \filldraw[black] (3,0) circle (1.5pt) node[anchor=north east]{};
            \filldraw[black] (3,1) circle (1.5pt) node[anchor=north east]{};
            \filldraw[black] (1,3) circle (1.5pt) node[anchor=north east]{};
            \filldraw[black] (0,3) circle (1.5pt) node[anchor=north east]{};
            \filldraw[black] (0,1) circle (1.5pt) node[anchor=north east]{};
            
            \filldraw[black] (1,1) circle (1.5pt) node[anchor=north east]{};
            \filldraw[black] (1,2) circle (1.5pt) node[anchor=north east]{};
            \filldraw[black] (2,1) circle (1.5pt) node[anchor=north east]{};
            
            \filldraw[black] (2,2) circle (1.5pt) node[anchor=north east]{};
            \filldraw[black] (0,2) circle (1.5pt) node[anchor=north east]{};
            \filldraw[black] (2,0) circle (1.5pt) node[anchor=north east]{};
        \end{scope}    
        \end{tikzpicture}

        \caption{Two different Minkowski decompositions of the facet~$F$.}
        \label{fig:F_decompositions}
    \end{figure}
There are two distinct Minkowski factorizations of~$F$: see Figure~\ref{fig:F_decompositions}. Corresponding to these two Minkowski factorizations, we consider two possible factors~$h_1$ and~$h_2$ for mutations with weight~$w$. In co-ordinates (as in Figure~\ref{fig:coefficients_h1}) where~$F$ is the convex hull of
    \[
        \text{$(1,0)$, $(3,0)$, $(3,1)$, $(1,3)$, $(0,3)$, and $(0,1)$}
    \]
    we have
    \begin{align*}
        h_1 &= (x+y)(1+y)(1+x)(1+x+y) \\
        h_2 &= (x+y+xy)(1+x+y)^2 
    \end{align*}
    Consider a normalised and centered Laurent polynomial~$f$ with~$\Newt(f) = P$, 
    \[
        f = \sum_{p \in N \cap P} a_p x^p    
    \]
    Insisting that~$f$ is mutable with respect to~$(w,h_1)$ imposes the divisibility condition discussed in Example~\ref{ex:divisibility}. This fixes the coefficients~$a_p$ of lattice points such that~$p \in F$ as in Figure~\ref{fig:coefficients_h1} and imposes no condition on other coefficients~$a_p$. Similarly, insisting that~$f$ is mutable with respect to~$(w,h_2)$ fixes the coefficients~$a_p$ of lattice points such that~$p \in F$ as in Figure~\ref{fig:coefficients_h2} and imposes no condition on other coefficients~$a_p$.
    
     \begin{figure}[ht]
        \centering
        \begin{tikzpicture}[every node/.style={scale=0.8},scale=0.8]
        
        \draw[black, very thick] (1,0) -- (3,0);
        \draw[black, very thick] (3,0) -- (3,1);
        \draw[black, very thick] (3,1) -- (1,3);
        \draw[black, very thick] (1,3) -- (0,3);
        \draw[black, very thick] (0,3) -- (0,1);
        \draw[black, very thick] (1,0) -- (0,1);
        
        \filldraw[black] (1,0) circle (1.5pt) node[anchor=north east]{$1$};
        \filldraw[black] (3,0) circle (1.5pt) node[anchor=north west]{$1$};
        \filldraw[black] (3,1) circle (1.5pt) node[anchor=south west]{$1$};
        \filldraw[black] (1,3) circle (1.5pt) node[anchor=south west]{$1$};
        \filldraw[black] (0,3) circle (1.5pt) node[anchor=south east]{$1$};
        \filldraw[black] (0,1) circle (1.5pt) node[anchor=north east]{$1$};
        
        \filldraw[black] (1,1) circle (1.5pt) node[anchor=east]{$4$};
        \filldraw[black] (1,2) circle (1.5pt) node[anchor=east]{$4$};
        \filldraw[black] (2,1) circle (1.5pt) node[anchor=east]{$4$};
        
        \filldraw[black] (2,2) circle (1.5pt) node[anchor=south west]{$2$};
        \filldraw[black] (0,2) circle (1.5pt) node[anchor=east]{$2$};
        \filldraw[black] (2,0) circle (1.5pt) node[anchor=north]{$2$};
        
        \end{tikzpicture}
        \caption{The coefficients~$a_p$,~$p \in F$, fixed by mutability with respect to~$(w,h_1)$.}
        \label{fig:coefficients_h1}
    \end{figure}
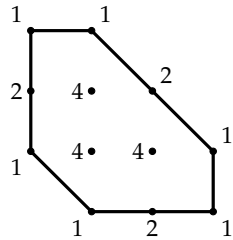
    
    \begin{figure}[ht]
        \label{h1FacetFact}
        \centering
        \begin{tikzpicture}[every node/.style={scale=0.8},scale=0.8]
        
        \draw[black, very thick] (1,0) -- (3,0);
        \draw[black, very thick] (3,0) -- (3,1);
        \draw[black, very thick] (3,1) -- (1,3);
        \draw[black, very thick] (1,3) -- (0,3);
        \draw[black, very thick] (0,3) -- (0,1);
        \draw[black, very thick] (1,0) -- (0,1);
        
        \filldraw[black] (1,0) circle (1.5pt) node[anchor=north east]{$1$};
        \filldraw[black] (3,0) circle (1.5pt) node[anchor=north west]{$1$};
        \filldraw[black] (3,1) circle (1.5pt) node[anchor=south west]{$1$};
        \filldraw[black] (1,3) circle (1.5pt) node[anchor=south west]{$1$};
        \filldraw[black] (0,3) circle (1.5pt) node[anchor=south east]{$1$};
        \filldraw[black] (0,1) circle (1.5pt) node[anchor=north east]{$1$};
        
        \filldraw[black] (1,1) circle (1.5pt) node[anchor=east]{$5$};
        \filldraw[black] (1,2) circle (1.5pt) node[anchor=east]{$5$};
        \filldraw[black] (2,1) circle (1.5pt) node[anchor=east]{$5$};
        
        \filldraw[black] (2,2) circle (1.5pt) node[anchor=south west]{$2$};
        \filldraw[black] (0,2) circle (1.5pt) node[anchor=east]{$2$};
        \filldraw[black] (2,0) circle (1.5pt) node[anchor=north]{$2$};
        
        \end{tikzpicture}
        \caption{The coefficients~$a_p$,~$p \in F$, fixed by mutability with respect to~$(w,h_2)$.}
        \label{fig:coefficients_h2}
    \end{figure}
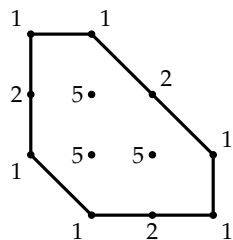

    Thus fixing, for each hexagonal facet~$F$ of~$P$, a choice of~$h_1$ or~$h_2$ defines a set~$S$ of four mutations, and there is a unique normalised, centered Laurent polynomial~$f$ with~$\Newt(f) = P$ such that~$f$ is mutable with respect to each element of~$S$. The Laurent polynomial~$f$ is a rigid maximally mutable Laurent polynomial~\cite[Definition~2.6]{CoatesKasprzykPittonTveiten2021}. In this way we obtain~$16$ rigid MMLPs, which fall into~$5$ equivalence classes under the action of~$\Aut(P)$.
\end{exa}

Given a Laurent polynomial~$f$, write 
\[
    S_f = \{ (w,h) : \text{$f$ is mutable with respect to~$(w,h)$}\}
\]
Conversely, given a set~$S$ of pairs~$(w, h)$ where~$w \in M$ is primitive and~$h \in \CC[w^\perp \cap N]$, write
\[
    L_P(S) = \left\{f \in \CC[N] :
    \begin{minipage}{0.55\textwidth}
        $f$ is normalised and centered,~$\Newt{f} = P$, and~$f$ is mutable with respect to~$(w,h_{gen})$ for all~$(w,h) \in S$
    \end{minipage}
    \right\}
\]
Here~$h_{gen}$ denotes the general normalised Laurent polynomial with the same Newton polytope as~$h$. The following is an immediate consequence of~\cite[Definition~2.6]{CoatesKasprzykPittonTveiten2021}.
\begin{pro} \label{pro:rigid}
    Let~$f$ be a normalised, centered Laurent polynomial with Newton polytope~$P$ and suppose that
    \[
        L_P(S_f) = \{f\}
    \]
    Then~$f$ is a rigid maximally mutable Laurent polynomial.
\end{pro}
We believe that the converse to Proposition~\ref{pro:rigid} also holds, that is, that~$f$ is a rigid MMLP if and only if~$L_P(S_f) = \{f\}$. In forthcoming work, Coates--Kasprzyk--Pitton use this, along with a large-scale computer algebra calculation, to classify rigid MMLPs in three variables with canonical Newton polytope~\cite{CoatesKasprzykPitton2022}, conditional on the converse to Proposition~\ref{pro:rigid}.
\section{Laurent Inversion and Towers of Bundles}\label{sec:towers}
In this section we give a conceptual interpretation of~\emph{Laurent inversion} in a special case. Laurent inversion is an algorithmic process for recovering a Fano manifold~$X$ from a Laurent polynomial that corresponds to~$X$ under mirror symmetry~\cite{CoatesKasprzykPrince2019}; this process may or may not succeed in any given example. The special case that we analyse, which in practice covers almost all cases in which Laurent inversion has been successfully applied, is where a certain algebraic variety involved, called the~\emph{shape variety}, is a tower of projective bundles. The discussion here reformulates and extends ideas that we learned from Charles Doran, Andrew Harder, and Thomas Prince~\cite{DoranHarder2016,Prince2020}.

\subsection{The Givental/Hori--Vafa Mirror}\label{sec:GHV}
Suppose that~$Y$ is a smooth Fano toric orbifold of dimension~$d$. Choosing a numbering of the rays of the fan~$\Sigma_Y$ for~$Y$ gives a short exact sequence
\begin{equation}
    \label{eq:ray-sequence}
    \xymatrix{
        0 \ar[r] &
        \LL \ar[r] &
        \ZZ^r \ar[r]^{\rho} &
        N \ar[r] & 
        0
    }
\end{equation}
where~$N$ is a~$d$-dimensional lattice and the map~$\rho$ is defined by the rays of~$\Sigma_Y$. Dualising gives a short exact sequence 
\begin{equation}
    \label{eq:dual-ray-sequence}
    \xymatrix{
        0 &
        \LL^\vee \ar[l] &
        (\ZZ^r)^\vee \ar[l]_D &
        M \ar[l] & 
        0 \ar[l]
    }
\end{equation}
where~$M = N^\vee$. There is a canonical isomorphism~$\LL^\vee \cong \Pic(Y)$, and the image of the standard basis for~$(\ZZ^r)^\vee$ under the map~$D$ gives a numbering~$D_1,\ldots,D_r$ of the toric divisors on~$Y$. The Givental/Hori--Vafa mirror~\cite{Givental1998,HoriVafa2000} to~$Y$ is the diagram
\begin{equation} \label{eq:mirror_to_ambient}
    \begin{aligned}
        \xymatrix{
            (\ZZ^r)^\vee \otimes \Cstar \ar[d]^{\pi_D} \ar[r]^-{W} &
            \CC \\
            \LL^\vee \otimes \Cstar
        }
    \end{aligned}
\end{equation}
where~$\pi_D$ is the fibration induced by~$D$ and~$W(x_1,\ldots,x_r) = x_1 + x_2 + \cdots + x_r$.

Suppose now that~$L_1,\ldots,L_c$ are line bundles over~$Y$, and that~$X \subset Y$ is a quasismooth, well-formed, Fano complete intersection\footnote{Quasismooth, well-formed weighted projective complete intersections have been studied by Iano-Fletcher~\cite{IanoFletcher2000}. See~\cite[Definition~25]{CortiHeuberger2017} for definitions applicable in our context.} defined by a general section of~$L_1 \oplus \cdots \oplus L_c$. Choosing disjoint subsets~$S_1,\ldots,S_c$ of~$\{1,2,\ldots,r\}$ such that
\begin{align}
    L_i = \sum_{j \in S_i} D_j && i \in \{1,2,\ldots,c\}
    \notag \\
\intertext{defines a Givental/Hori--Vafa mirror to~$X$. This is the subvariety of~\eqref{eq:mirror_to_ambient} defined by}
     \label{eq:mirror_to_ci}
    \sum_{j \in S_i} x_j = 1 && i \in \{1,2,\ldots,c\}.
\end{align}

\subsection{Towers of bundles}\label{sec:towers_of_bundles}
The equations~\eqref{eq:mirror_to_ci} define a codimension-$c$ subvariety of the total space of the fibration~$\pi_D$ in~\eqref{eq:mirror_to_ambient}, which we call the~\emph{GHV locus}. We will now choose some extra data that allows us to define a toric partial compactification of the GHV locus. This partial compactification arises from an action of~$(\Cstar)^c$ on the total space~$(\ZZ^r)^\vee \otimes \Cstar$ of~$\pi_D$: we realise the codimension-$c$ locus defined by~\eqref{eq:mirror_to_ci} as an open set inside a toric variety~$\cZ = \big((\ZZ^r)^\vee \otimes \CC \big) \sslash (\Cstar)^c$. The ray sequence for the toric variety~$\cZ$ and the dual ray sequence~\eqref{eq:dual-ray-sequence} for~$Y$ fit together as follows:
\begin{equation}
    \label{eq:foreshadow}
\begin{aligned}
    \xymatrix{
        & 0 \\
        && N_\cZ \ar[ul] \\
        0 &
        \LL^\vee \ar[l] &&
        (\ZZ^r)^\vee \ar[ll]_D \ar[ul] &&
        M \ar[ll]_{\rho^T} & 
        0 \ar[l] \\
        &&&& \ZZ^c \ar[ul] \ar@{..>}[ur] \\
        &&&&& 0 \ar[ul]
    }
\end{aligned}
\end{equation}

To this end, let~$S_0 = \{1,2,\ldots,r\} \setminus S_1 \cup \cdots \cup S_c$. We will consider a map\footnote{The transpose of~$\zeta$ will be the dotted arrow in~\eqref{eq:foreshadow}. The existence of~$\zeta$ ensures that the action of~$(\Cstar)^c$ on~$(\ZZ^r)^\vee \otimes \CC$ preserves the fibers of~$\pi_D$ in~\eqref{eq:mirror_to_ambient}.}
\[
    \xymatrix{
        \zeta: N \ar[r] & (\ZZ^c)^\vee
    }
\]
and write~$\zeta_i$ for the image under~$\zeta$ of the~$i$th ray~$\rho_i$ of the fan~$\Sigma_Y$. We will suppose that~$\zeta$ defines the weight matrix for a tower of projective bundles, that is, if~$e_j$ denotes the~$j$th standard basis vector for~$\ZZ^c$ then
\begin{align*}
    \zeta_k(e_j) = 
    \begin{cases}
        w_{jk} & \text{$j<i$ or~$i=0$} \\
        1 & j = i \\
        0 & j > i 
    \end{cases}
    && \text{where~$k \in S_i$ and~$w_{jk} \leq 0$.}   
\end{align*}
We call such a map~$\zeta$ a~\emph{tower of bundles} for the complete intersection~$X \subset Y$.

If we permute~$\{1,2,\ldots,r\}$ such that~$S_0,\ldots,S_c$ occur in that order, i.e.~whenever~$i \in S_k$ and~$j \in S_l$ with~$k<l$ we have~$i<j$, then the matrix of the composition~$\zeta \circ \rho$ takes the form
\begin{equation}
    \label{eq:weight_matrix}
    \left(
    \begin{array}{*{13}c}
        * & \cdots & * & 1 & \cdots & 1 & * & \cdots & * &        &  * & \cdots & * \\
        * & \cdots & * & 0 & \cdots & 0 & 1 & \cdots & 1 & \cdots & * & \cdots & *\\ 
        * & \cdots & * & 0 & \cdots & 0 & 0 & \cdots & 0 & \cdots & \vdots & \cdots & \vdots \\ 
          & \ddots &   &         &   &   &   & \ddots &   &        &  * & \cdots & *\\ 
        * & \cdots & * & 0 & \cdots & 0 & 0 & \cdots & 0 & \cdots & 1 & \cdots & 1\\      
    \end{array}
    \right)
\end{equation}
where~$*$ denotes a non-positive integer. Let~$n = |S_1|+\cdots+|S_c|$. The last~$n$ columns of the matrix above give the weight matrix for an action of~$(\Cstar)^c$ on~$\CC^n$ such that the GIT quotient~$\CC^n \sslash (\Cstar)^c$, with stability condition~$(1,1,\ldots,1)$, is a tower of projective bundles~$\PP$. Each of the first~$|S_0|$ columns defines a line bundle~$E_i \to \PP$ such that the dual bundle~$E_i^\vee$ is nef. 

The map~$\zeta$ is closely related to Doran--Harder's notion of~\emph{amenable collection}~\cite{DoranHarder2016}, and the tower of projective bundles~$\PP$ will play the role of the~\emph{shape variety}~$Z$ from~\cite{CoatesKasprzykPrince2019}. Thus the discussion which follows gives a geometric interpretation of amenable collections, and clarifies the relationship between the shape variety and Givental/Hori--Vafa mirror symmetry.

\subsection{Partially compactifying the total space of the Givental/Hori--Vafa mirror}\label{sec:partially_compactify_total_space}
Let~$\zeta$ be a tower of bundles for the complete intersection~$X \subset Y$. Dualising the map~$\zeta \circ \rho : \ZZ^r \to (\ZZ^c)^\vee$ gives a map~$\ZZ^c \to (\ZZ^r)^\vee$, and hence a map~$\ZZ^c \otimes \Cstar \to (\ZZ^r)^\vee \otimes \Cstar$. Thus we obtain an action of~$(\Cstar)^c$ on~$(\ZZ^r)^\vee \otimes \CC$. Consider the GIT quotient~$\cZ = \big((\ZZ^r)^\vee \otimes \CC \big) \sslash (\Cstar)^c$, with respect to the stability condition~$(1,1,\ldots,1)$. This is the total space
\[
E = \bigoplus_{i \in S_0} E_i
\]
of a direct sum of anti-nef line bundles over~$\PP$, defined by the first~$|S_0|$ columns of the weight matrix~\eqref{eq:weight_matrix}. In this section we show that~$E$ is a partial compactification of the GHV locus~\eqref{eq:mirror_to_ci}. 

\begin{dfn}
    \label{dfn:theta}
    We define functions~$\theta_k \colon (\ZZ^r)^\vee \otimes \CC \to \CC$ recursively by
    \begin{align*}
        \theta_k(x) &= \sum_{j \in S_k} \left( \prod_{m=1}^{k-1} \theta_m(x)^{-w_{mj}} \right) x_j && k \in \{1,2,\ldots,c\}
    \end{align*}
    where~$x = (x_1,x_2,\ldots,x_r) \in (\ZZ^r)^\vee \otimes \CC$. In particular,~$\theta_1(x) = \sum_{j \in S_1} x_j$.
\end{dfn}

\begin{pro} \ \label{pro:partial_compactification}
    \begin{enumerate}
        \item \label{item:partial_compactification_1} Let~$g \in (\Cstar)^c$ and~$x \in (\ZZ^r)^\vee$. The function~$\theta_k$ satisfies 
        \[
            \theta_k(g x) = \chi_k(g) \theta_k(x)
        \]
        where~$\chi_k$ is the~$k$th standard character of~$(\Cstar)^c$. 
        \item \label{item:partial_compactification_2} The function~$\theta_k$ determines a section of the line bundle~$L_k \to E$ defined by the character~$\chi_k$.
        \item \label{item:partial_compactification_3} The open set~$U$ in~$E$ defined by 
        \[
            \big(\theta_1 \ne 0, \theta_2 \ne 0, \ldots, \theta_c \ne 0,x_1 \ne 0,x_2 \ne 0, \ldots, x_r \ne 0\big)
        \]
        is isomorphic to the GHV locus~\eqref{eq:mirror_to_ci}.
    \end{enumerate}
\end{pro}
\begin{proof}
Part~\eqref{item:partial_compactification_1} here is a straightforward calculation. Part~\eqref{item:partial_compactification_2} is a restatement of part~\eqref{item:partial_compactification_1}. For part~\eqref{item:partial_compactification_3}, consider~$(x_1,\ldots,x_r) \in (\ZZ^r)^\vee \otimes \CC$ such that~$x_1 \ne 0$, \ldots,~$x_r \ne 0$ and~$\theta_1 \ne 0$, \ldots,~$\theta_c \ne 0$. The image of~$(x_1,\ldots,x_r)$ under the action of~$g = (\theta_1^{-1},\ldots,\theta_c^{-1})$ is~$(y_1,\ldots,y_r)$, where
    \begin{equation}
        \label{eq:y_j}
        y_j = 
        \begin{cases}
            x_j \theta_k^{-1} \prod_{m=1}^{k-1} \theta_m^{-w_{mj}} & k \ne 0 \\
            x_j \prod_{m=1}^{c} \theta_m^{-w_{mj}} & k =0
        \end{cases}
    \end{equation}                
and~$j \in S_k$. Part~\eqref{item:partial_compactification_1} now implies that~$y_j$ depends only on the~$(\Cstar)^c$-orbit of~$(x_1,\ldots,x_r)$. Since 
    \begin{align*}
        \sum_{j \in S_k} y_j = 1 && j \in \{1,2,\ldots,c\}
    \end{align*}
we see that mapping~$(x_1,\ldots,x_r)$ to~$(y_1,\ldots,y_r)$ defines a map~$\phi$ from the open set~$U$ to the locus~\eqref{eq:mirror_to_ci}. Setting~$x_j = y_j$ defines an inverse to~$\phi$; this completes the proof. 
\end{proof}

\begin{pro}\label{pro:extends_holomorphically}
The function~$W$ from~\eqref{eq:mirror_to_ci}, which in view of Proposition~\ref{pro:partial_compactification} can be regarded as a function on~$U$, extends holomorphically across the locus
\[
        (\theta_1 = 0, \ldots, \theta_k=0)   
\]
in~$E$.
\end{pro}

\begin{proof}
By construction
\[
        W = \sum_{j \in S_0} y_j + c
\]
on~$U$. The result now follows from~\eqref{eq:y_j}, because~$w_{mj} \leq 0$ for all~$m$ and~$j$.
\end{proof}

\subsection{Laurent polynomial mirrors}\label{sec:laurent_polynomial_mirrors}
In the approach to the classification of Fano varieties pioneered by Corti and Golyshev, Fano varieties of dimension~$n$ conjecturally correspond to equivalence classes of Laurent polynomials in~$n$~variables~\cite{CoatesCortiGalkinGolyshevKasprzyk2013}. This correspondence is an instance of mirror symmetry. If~$X$ is a Fano toric variety then a Laurent polynomial~$f$ that corresponds to~$X$ under mirror symmetry can be obtained from the Givental/Hori--Vafa mirror by restriction to a fiber, as follows. In the notation of~\S\ref{sec:GHV} we take~$X=Y$ and~$c=0$, so that the set of equations~\eqref{eq:mirror_to_ci} is empty. The Givental/Hori--Vafa mirror to~$X$ is then the diagram~\eqref{eq:mirror_to_ambient}. The Laurent polynomial~$f$ arises by restricting the superpotential~$W$ to the fiber of~$\pi_D$ over the identity element in~$\LL^\vee \otimes \Cstar$. From the exact sequence~\eqref{eq:dual-ray-sequence} we see that this fiber is canonically identified with~$M \otimes \Cstar$. The restriction of~$W$ to this fiber is given by 
\[
    f = \sum_{i=1}^r x^{\rho_i} 
\]
where~$\rho_i \in N$ is the~$i$th ray of the fan for~$X$. Here~$x^{\rho_i}$ arises as restriction of the function~$x_i$ to the fiber of~$\pi_D$ over the identity element; put differently,~$x^{\rho_i}$ arises as the image of~$x_i$ (which is the function on~$(\ZZ^r)^\vee \otimes \Cstar$ given by the~$i$th standard basis element in~$\ZZ^r$) under the ray map~$\rho$ in~\eqref{eq:ray-sequence}. Since~$\rk M = \dim X = n$, we see that~$f$ is a Laurent polynomial in~$n$ variables.

If~$X$ is a Fano toric complete intersection (as opposed to a toric variety) then the process of obtaining a Laurent polynomial from the Givental/Hori--Vafa mirror is more involved. It amounts to choosing a~\emph{torus chart} -- an open set birational to an~$n$-dimensional torus -- on the fiber of the locus~\eqref{eq:mirror_to_ci} over the identity element, such that the restriction of the superpotential~$W$ to this open set is a Laurent polynomial. As we will explain in the next section, one way to construct such a torus chart arises from a tower of bundles that satisfy an additional condition.

\subsection{Partially compactifying the fiber of the Givental/Hori--Vafa mirror}\label{sec:partially_compactify_fiber}
In~\S\ref{sec:partially_compactify_total_space} we described how a tower of bundles gives rise to a partial compactification of the GHV locus. In this section we give a refinement of this construction which preserves the fibration structure given by~$\pi_D$ in~\eqref{eq:mirror_to_ci}. That is, we construct a fiberwise partial compactification of the GHV locus. As promised, this also gives a torus chart on the fiber of~$\pi_D$ over the identity element of~$\LL^\vee \otimes \Cstar$. The key ingredient is a tower of bundles that contains a basis.

\begin{dfn}
    \label{dfn:contains_basis}
    We say that a tower of bundles for the toric complete intersection~$X \subset Y$~\emph{contains a basis} if the set of toric divisors~$\{D_i : i \in S_0\}$ contains a basis for the lattice~$\Pic(Y)$.
\end{dfn}

Consider a tower of bundles for the toric complete intersection~$X \subset Y$ that contains a basis. Let~$B \subset S_0$ be such that~$\{D_j : j \in B\}$ is a basis for~$\Pic(Y)$. Without loss of generality we may permute~$\{1,2,\ldots,r\}$ such that whenever~$i \in S_k$ and~$j \in S_l$ with~$k<l$ we have~$i<j$, and that~$B = \{1,2,\ldots,|B|\}$. The ray sequence~\eqref{eq:ray-sequence} and the tower of bundles together give a diagram:
\begin{equation}
    \label{eq:ray-sequence-plus-tower}
    \begin{aligned}
        \xymatrix{
            0 \ar[r] &
            \LL \ar[rr]^-{D^T} && 
            \ZZ^r \ar[rr]^-\rho \ar[rd] &&
            N \ar[r] \ar[ld]^-\zeta &
            0 \\
            &&&&(\ZZ^c)^\vee
        }
    \end{aligned}   
\end{equation}
Write~$B'$ =~$\{1,2,\ldots,r\} \setminus B$, so that~$\ZZ^r = \ZZ^B \oplus \ZZ^{B'}$, and let~$p_B \colon \ZZ^r \to \ZZ^B$ denote the projection. Our assumptions about the basis guarantee both that~$p_B \circ D^T$ is an isomorphism and that~$\{\rho_j : j \in B'\}$ is a basis for~$N$. Thus there are (unique) identifications of~$\LL$ with~$\ZZ^B$ and~$N$ with~$\ZZ^{B'}$ such that in the diagram
\begin{equation}
    \label{eq:ray-sequence-plus-tower-with-splitting}
    \begin{aligned}
        \xymatrix{
            0 \ar[r] &
            \ZZ^B \ar[rr]^-{D^T} && 
            \ZZ^r \ar[rr]^-\rho \ar[rd] &&
            \ZZ^{B'} \ar[r] \ar[ld]^-\zeta &
            0 \\
            &&&&(\ZZ^c)^\vee
        }
    \end{aligned}   
\end{equation}
induced from~\eqref{eq:ray-sequence-plus-tower}, both~$p_B \circ D^T$ and~$\rho|_{\ZZ^{B'}}$ are identity maps. Thus a tower of bundles with a basis gives splittings in two different senses:
\begin{enumerate}
    \item[(a)] \hypertarget{item:splitting_a}{} a splitting~$\ZZ^{B'} \to \ZZ^r$ of the ray map~$\rho$, and hence of the ray sequence~\eqref{eq:ray-sequence}; 
    \item[(b)] \hypertarget{item:splitting_b}{} a splitting~$N \otimes \Cstar \cong \ZZ^{B'} \otimes \Cstar$ of the torus~$N \otimes \Cstar$.
\end{enumerate}
Splitting (\hyperlink{item:splitting_b}{b}) here is equivalent, by duality, to:
\begin{enumerate}
    \item[(b$'$)] \hypertarget{item:splitting_b_prime}{} a splitting~$M \otimes \Cstar \cong (\ZZ^{B'})^\vee \otimes \Cstar$ of the torus~$M \otimes \Cstar$.
\end{enumerate}
Recall that~$M \otimes \Cstar$ is the fiber of the Givental/Hori--Vafa mirror~\eqref{eq:mirror_to_ambient} for the ambient space~$Y$.

Consider now the mirror fibration~$\pi_D$ from~\eqref{eq:mirror_to_ambient}. This is a principal~$M \otimes \Cstar$-bundle
\begin{equation}
    \label{eq:principal_torus_bundle}
    \begin{aligned}
        \xymatrix{
            M \otimes \Cstar \ar[r] & (\ZZ^r)^\vee \otimes \Cstar \ar[d]^-{\pi_D} \\
            & \LL^\vee \otimes \Cstar
        }
    \end{aligned}
\end{equation}
and splitting (\hyperlink{item:splitting_b_prime}{b$'$}) identifies this with a principal~$(\ZZ^{B'})^\vee \otimes \Cstar$-bundle. Since it is split, the torus~$(\ZZ^{B'})^\vee \otimes \Cstar$  acts canonically on the vector space~$(\ZZ^{B'})^\vee \otimes \CC = (\CC^{B'})^\vee$. We form the associated vector bundle to~\eqref{eq:principal_torus_bundle} with fiber~$(\CC^{B'})^\vee$; this is a vector bundle~$\cV \to \LL^\vee \otimes \Cstar$ of rank~$|B'|$. The vector bundle~$\cV$ carries a fiberwise action of~$(\Cstar)^c$, given by dualising the map~$\zeta$ in~\eqref{eq:ray-sequence-plus-tower-with-splitting}.

\begin{dfn}
    Let~$\cZ$ denote the GIT quotient~$\cV \sslash (\Cstar)^c$ with stability condition~$(1,1,\ldots,1)$, and let~$Z$ denote the fiber of~$\cZ \to  \LL^\vee \otimes \Cstar$ over the identity element~$e \in \LL^\vee \otimes \Cstar$.
\end{dfn}

\begin{pro}
    \ 
    \begin{enumerate}
        \item \label{item:anonymous_1} $\cZ$ is an open subset of the partial compactification~$E$ defined in~\S\ref{sec:partially_compactify_total_space};
        \item \label{item:anonymous_2} $\cZ$ is a fiberwise partial compactification of the GHV locus~\eqref{eq:mirror_to_ci};
        \item \label{item:anonymous_3} the fiber~$Z$ is a toric variety; 
        \item \label{item:anonymous_4} if~$S_0 = B$ then~$Z$ is isomorphic to~$\PP$, the tower of projective bundles defined in~\S\ref{sec:towers_of_bundles};
        \item \label{item:anonymous_5} if~$S_0$ strictly contains~$B$ then~$Z$ is isomorphic to the total space of the direct sum of anti-nef line bundles over~$\PP$ defined by the columns of~\eqref{eq:weight_matrix} indexed by~$S_0 \setminus B$. 
    \end{enumerate}
\end{pro}

\begin{proof}
The splitting (\hyperlink{item:splitting_a}{a}) shows that the principal~$M \otimes \Cstar$-bundle~\eqref{eq:principal_torus_bundle} is in fact trivial:
    \[
        \xymatrix{
            (\ZZ^{B'})^\vee \otimes \Cstar \ar[r] & (\ZZ^r)^\vee \otimes \Cstar \ar[d]^-{\pi_D} \\
            & (\ZZ^B)^\vee \otimes \Cstar
        }
    \]
Thus~$(\ZZ^r)^\vee \otimes \CC$ contains the total space of~$\cV$ as an open set. The embedding of~$\cV$ into~$(\ZZ^r)^\vee \otimes \CC$ respects the action of~$(\Cstar)^c$, so this proves~\eqref{item:anonymous_1}.

To prove~\eqref{item:anonymous_2}, we argue exactly as in Proposition~\ref{pro:partial_compactification}. Fix~$q \in \LL^\vee \otimes \Cstar$ and consider the fiber~$\cV_q$ of the vector bundle~$\cV$ over~$q$. The splitting (\hyperlink{item:splitting_b_prime}{b$'$}) gives rise to distinguished co-ordinates~$(x_j : j \not \in B)$, on~$\cV_q$, and we define\footnote{Or, equivalently, we could restrict the functions~$\theta_k: (\ZZ^r)^\vee \otimes \CC \to \CC$ from Definition~\ref{dfn:theta} to~$\cV_q$ via the embedding~$\cV \to (\ZZ^r)^\vee \otimes \CC$ just discussed.} functions~$\theta_k: \cV_q \to \CC$ exactly as in Definition~\ref{dfn:theta}:
    \begin{align*}
        \theta_k(x) &= \sum_{j \in S_k} \left( \prod_{m=1}^{k-1} \theta_m(x)^{-w_{mj}} \right) x_j && k \in \{1,2,\ldots,c\}
    \end{align*}
    Consider functions~$y_j$ defined as in~\eqref{eq:y_j}:
    \begin{align*}
        y_j = 
        \begin{cases}
            x_j \theta_k^{-1} \prod_{m=1}^{k-1} \theta_m^{-w_{mj}} & k \ne 0 \\
            x_j \prod_{m=1}^{c} \theta_m^{-w_{mj}} & k =0
        \end{cases}
        && j \not \in B
    \end{align*}
    where~$k$ is such that~$j \in S_k$. The open set~$\tilde{U} \subset \cV_q$ defined by
    \[
    \big(\theta_1 \ne 0, \theta_2 \ne 0, \ldots, \theta_c \ne 0) \cap (x_j \ne 0 : j \not \in B\big)
    \]
    projects to an open set~$U$ in the quotient~$\cV_q \sslash (\Cstar)^c$, and the map
    \begin{equation}
        \label{eq:phi_tilde}
        \begin{aligned}
        \tilde{\phi} : \tilde{U} & \longrightarrow (\ZZ^{B'})^\vee \otimes \Cstar \\
        (x_j : j \not \in B) & \longmapsto (y_j : j \not \in B)
        \end{aligned}
    \end{equation}
    descends to give a well-defined map~$\phi$ from~$U$ to the subset of~$(\ZZ^{B'})^\vee \otimes \Cstar$  defined by
    \begin{align*}
        \sum_{j \in S_i} y_j = 1 && i \in \{1,2,\ldots,c\}.
    \end{align*}
The map~$\phi$ is an isomorphism, with inverse given by setting~$x_j = y_j$,~$j \not \in B$. Thus the Givental/Hori--Vafa locus~\eqref{eq:mirror_to_ci} embeds into~$\cZ$ as an open set, and this embedding exhibits~$\cZ$ as a fiberwise partial compactification of~\eqref{eq:mirror_to_ci}.
    
Part~\eqref{item:anonymous_3} is obvious, as~$Z$ is the GIT quotient~$\cV_e \sslash (\Cstar)^k$. To identify this quotient, we examine the weights of the~$(\Cstar)^k$-action on the fiber. These are the entries in the matrix of the map~$\zeta$ in diagram~\eqref{eq:ray-sequence-plus-tower-with-splitting}, with respect to the standard bases for~$\ZZ^{B'}$ and~$(\ZZ^c)^\vee$. This matrix is given by the last~$|B'|$ columns of~\eqref{eq:weight_matrix}. If~$S_0 = B$ then this is
    \[
        \left(
        \begin{array}{*{12}c}
            1 & \cdots & 1 & * & \cdots & * & \cdots & * & \cdots & * \\
            0 & \cdots & 0 & 1 & \cdots & 1 & \cdots & * & \cdots & *\\ 
            0 & \cdots & 0 & 0 & \cdots & 0 & \cdots & \vdots & \cdots & \vdots \\ 
           &        &   &   & \ddots &   &        & * & \cdots & *\\ 
            0 & \cdots & 0 & 0 & \cdots & 0 & \cdots & 1 & \cdots & 1\\      
        \end{array}
        \right)
    \]
which proves~\eqref{item:anonymous_4}. Otherwise there are an additional~$|S_0 \setminus B|$ leading columns, all of which contain non-positive entries; this proves~\eqref{item:anonymous_5}.
\end{proof}

\subsection{Laurent polynomials and scaffoldings from towers of bundles}\label{sec:scaffoldings_from_towers}
In~\S\ref{sec:partially_compactify_fiber} we constructed a toric partial compactification~$Z$ of the fiber of the Givental/Hori--Vafa mirror to a toric complete intersection~$X \subset Y$. In this section we explain how this gives rise to a Laurent polynomial~$f$ that corresponds to~$X$ under mirror symmetry (see~\S\ref{sec:laurent_polynomial_mirrors}). The Laurent polynomial~$f$ arises as a function on the dense torus in~$Z$, and comes equipped with a~\emph{scaffolding}~\cite{CoatesKasprzykPrince2019}. This is a decomposition of~$f$ into summands, called~\emph{struts}, of a specific form. 

The superpotential~$W$ in the Givental/Hori--Vafa mirror restricts to the locus~\eqref{eq:mirror_to_ci} to give
\[
    W = \sum_{j \in S_0} y_j + c
\]
(cf.~Proposition~\ref{pro:extends_holomorphically}), and restricts to the fiber over~$e \in \LL^\vee \otimes \Cstar$ to give
\[
    W = \sum_{j \in S_0} y^{\rho_j} + c
\]
where~$\rho_j \in N$ is the~$j$th ray of the fan for~$Y$. To obtain a meromorphic function on~$Z$, we first pull~$W$ back to the open set~$\tilde{U} \subset \cV_e$ along the map~\eqref{eq:phi_tilde}, finding
\begin{equation}
    \label{eq:pull_back_lift}
    {\tilde{\phi}}^* (W) = \sum_{j \in S_0} x^{\rho_j} \prod_{m=1}^c \theta_m^{-{w_{mj}}} + c
\end{equation}
The function~${\tilde{\phi}}^* (W)$ is invariant under the action of~$(\Cstar)^c$ by construction -- and indeed we see that each summand on the right-hand side of~\eqref{eq:pull_back_lift} is homogeneous of weight zero -- and so~\eqref{eq:pull_back_lift} descends to give a well-defined function~$\phi^*(W)$ on~$U \subset Z$. As in Proposition~\ref{pro:extends_holomorphically}, the function~$f \coloneqq \phi^* W$ extends holomorphically across the locus
\[
    (\theta_1=0,\ldots,\theta_k=0)
\]
in~$Z$, and thus defines a holomorphic function on the dense torus in~$Z$.

We want to regard the function~$f$ as a Laurent polynomial. That is, we want to construct a splitting of the dense torus  
\[
    T_Z = \left(\big(\ZZ^{B'}\big)^\vee \otimes \Cstar\right) / (\Cstar)^c
\] 
in~$Z$. Such a splitting will give distinguished co-ordinates on~$T$, and expressing~$f$ in terms of these co-ordinates will yield a Laurent polynomial. Choose a set~$F$ made up of one element from each~$S_i$,~$i \in \{1,2,\ldots,c\}$. Then 
\[
    T_Z \cong \big(\ZZ^{B' \setminus F}\big)^\vee \otimes \Cstar.
\]
To express~$f$ in these co-ordinates, we take the expression~\eqref{eq:pull_back_lift} and set~$x_j = 1$ for all~$j \in F$. The result is a Laurent polynomial in variables 
\[
    \{x_j : j \in B', j \not \in F\}
\]
Each summand~$x^{\rho_j} \prod_{m=1}^c \theta_m^{-{w_{mj}}}$ in~\eqref{eq:pull_back_lift} gives a strut, and so~$f$ comes with a distinguished scaffolding.

\begin{rem}
    In the original work on Laurent inversion~\cite{CoatesKasprzykPrince2019}, the struts in a scaffolding are polytopes of sections of nef line bundles on a toric variety called the shape variety. Here we consider struts as specific sections of line bundles on the toric variety~$Z$. That these line bundles are nef follows from the fact that~$\prod_{m=1}^c \theta_m^{-{w_{mj}}}$ is a section of the line bundle~$E_j^\vee \to Z$ where the dual bundle~$E_j \to Z$ is defined by the~$j$th column of the weight matrix~\eqref{eq:weight_matrix}. Since~$w_{mj} \leq 0$ for all~$m$ and~$j$, the line bundle~$E_j^\vee \to Z$ is nef.
\end{rem}

\begin{rem}
    \label{rem:binomial}
    Recall that, if a Laurent polynomial~$f$ corresponds under mirror symmetry to a Fano variety~$X$, then it is expected that there is a qG\nobreakdash-degeneration with general fiber~$X$ and special fiber the toric variety~$X_f$. In our situation Doran--Harder have constructed an embedded degeneration~\cite{DoranHarder2016} of the complete intersection~$X$ to the toric subvariety of~$Y$ defined, in Cox co-ordinates~$(z_i)_{i=1}^r$, by the binomials
    \begin{align*}
        \prod_{j \in S_i} z_j = \prod_{j \not \in S_i} z_j^{-w_{ij}} && i \in \{1,2,\ldots,c\}
    \end{align*}
    This is the expected degeneration of~$X$ to~$X_f$~\cite[Proposition~12.2]{CoatesKasprzykPrince2019}.
\end{rem}

\begin{exa}
    \label{exa:dP6}
    Let~$X$ denote a complete intersection of type~$(1,1) \cdot (1,1)$ in~$Y=\PP^2 \times \PP^2$, and write the ray sequence~\eqref{eq:ray-sequence} for~$Y$ as
    \[
        \xymatrix{
            0 \ar[r] &
            \ZZ^2 \ar[rr]^-{
                \begin{pmatrix}
                    1 & 0 \\
                    0 & 1 \\
                    1 & 0 \\
                    0 & 1 \\
                    1 & 0 \\
                    0 & 1
                \end{pmatrix}
            } &&
            \ZZ^6 \ar[rrrrr]^{ 
                \begin{pmatrix}
                    -1 &  0 & 1 & 0 & 0 & 0 \\
                    -1 &  0 & 0 & 0 & 1 & 0 \\
                     0 & -1 & 0 & 1 & 0 & 0 \\
                     0 & -1 & 0 & 0 & 0 & 1 
                \end{pmatrix}
            } &&&&&
            \ZZ^4 \ar[r] & 
            0
        }
    \]
    Set~$S_0 = \{1,2\}$,~$S_1 = \{3,4\}$,~$S_2 = \{5,6\}$, and consider a weight matrix~\eqref{eq:weight_matrix} for a tower of bundles:
    \[
        \begin{pmatrix}
            * & * & 1 & 1 & 0 & -a \\
            * & * & 0 & 0 & 1 & 1
        \end{pmatrix}
    \]
    Solving for the leftmost two columns, using the fact that the weight matrix left-annihilates the first matrix in the ray sequence, yields
    \[
        \begin{pmatrix}
            -1 & a-1 & 1 & 1 & 0 & -a \\
            -1 & -1 & 0 & 0 & 1 & 1
        \end{pmatrix}
    \]
    and since we need~$a-1 \leq 0$ and~${-a} \leq 0$ it follows that~$a$ must be either~$0$ or~$1$. Both choices give a tower of bundles with basis, and in each case~$S_0 = B$. We have:
    \begin{align*}
        \theta_1 &= x_3 + x_4 \\
        \theta_2 &= x_5 + (x_3+x_4)^a x_6
    \end{align*}
    The pullback~\eqref{eq:pull_back_lift} is
    \[
        \frac{(x_3+x_4)(x_5 + (x_3+x_4)^a x_6)}{x_3 x_5} +
        \frac{(x_3+x_4)^{1-a}(x_5 + (x_3+x_4)^a x_6)}{x_2 x_4} + 2
    \]
    and we regard this as a Laurent polynomial by setting
    \begin{align*}
        x_3 = 1 && x_4 = x && x_5 = 1 && x_6 = y
    \end{align*}
    obtaining either
    \[
        f = (1+x)(1+y) + \frac{(1+x)(1+y)}{xy} + 2
    \]
    if~$a=0$, or 
    \[
        f = (1+x)(1+(1+x)y) + \frac{1+(1+x)y}{xy} + 2
    \]
    if~$a=1$. The two Newton polytopes are shown, together with the Newton polytopes of the struts, in Figure~\ref{fig:tower_example}. It is striking that the two Laurent polynomials, and scaffoldings, that result differ by a mutation~\cite{AkhtarCoatesGalkinKasprzyk2012}.
    \begin{figure}[ht]
        \centering
        \begin{tikzpicture}
        
        
        \draw[black, very thick] (-1,-1) -- (-1,0) -- (0,1) -- (1,1) -- (1,0) -- (0,-1) -- cycle;
        
        \filldraw[black] (-1,-1) circle (1.5pt) node[anchor=north east] {$1$};
        \filldraw[black] (-1,0) circle (1.5pt) node[anchor=east] {$1$};
        \filldraw[black] (0,1) circle (1.5pt) node[anchor=south] {$1$};
        \filldraw[black] (1,1) circle (1.5pt) node[anchor=south west] {$1$};
        \filldraw[black] (1,0) circle (1.5pt) node[anchor=west] {$1$};
        \filldraw[black] (0,-1) circle (1.5pt) node[anchor=north west] {$1$};

        \filldraw[color=MidnightBlue!60, fill=MidnightBlue!65, very thick] (0+1/10,0+1/10) -- (0+1/10,1-1/10) -- (1-1/10,1-1/10) -- (1-1/10, 0+1/10) -- cycle;
        
        \filldraw[color=BurntOrange!60, fill=BurntOrange!65, very thick] (-1+1/10,-1+1/10) -- (-1+1/10,0-1/10) -- (0-1/10,0-1/10) -- (0-1/10, -1+1/10) -- cycle;
        
        \draw[black, thick] (-0.05,0.05) -- (0.05,-0.05);
        \draw[black, thick] (-0.05,-0.05) -- (0.05,0.05);
        
        \end{tikzpicture}
        \qquad
        \qquad
        \qquad
        \begin{tikzpicture}
        
        
        \draw[black, very thick] (-1,-1) -- (-1,0) -- (0,1) -- (2,1) -- (1,0) -- cycle;
        
        \filldraw[black] (-1,-1) circle (1.5pt) node[anchor=north east] {$1$};
        \filldraw[black] (-1,0) circle (1.5pt) node[anchor=east] {$1$};
        \filldraw[black] (0,1) circle (1.5pt) node[anchor=south] {$1$};
        \filldraw[black] (1,1) circle (1.5pt) node[anchor=south] {$2$};
        \filldraw[black] (2,1) circle (1.5pt) node[anchor=south west] {$1$};
        \filldraw[black] (1,0) circle (1.5pt) node[anchor=north west] {$1$};

        \filldraw[color=MidnightBlue!60, fill=MidnightBlue!65, very thick] (0+1/10,0+1/10) -- (0+1/10,1-1/10) -- (2-5/20,1-1/10) -- (1-1/20, 0+1/10) -- cycle;
        
        \filldraw[color=BurntOrange!60, fill=BurntOrange!65, very thick] (-1+1/10,-1+3/20) -- (-1+1/10,0-1/10) -- (0-3/20,0-1/10) -- cycle;
        
        \draw[black, thick] (-0.05,0.05) -- (0.05,-0.05);
        \draw[black, thick] (-0.05,-0.05) -- (0.05,0.05);
        
        \end{tikzpicture}
        \caption{The scaffoldings arising from towers of bundles in Example~\ref{exa:dP6}.}
        \label{fig:tower_example}
        \end{figure}
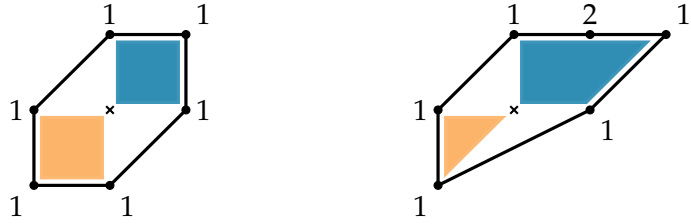
\end{exa}

\subsection{Reconstructing toric complete intersections from scaffoldings}\label{sec:reconstruction}
Laurent inversion in this context amounts to the assertion that the construction in~\S\ref{sec:scaffoldings_from_towers} is reversible: that we can reconstruct~$X \subset Y$ and the tower of bundles from the resulting Laurent polynomial~$f$ and its scaffolding. This is clear. From the scaffolding one can read off the functions~$\theta_k(x)$, or more precisely the restrictions
\begin{align*}
    \theta_k(x)\big|_{\text{$x_j=1$ for~$j \in F$}} && k \in \{1,2,\ldots,c\}.
\end{align*}
In particular this determines the weights~$w_{jk}$ in the weight matrix~\eqref{eq:weight_matrix} with~$k$ in~$S_1 \cup S_2 \cup \cdots \cup S_c$. These are the shaded weights here:
\[
    \left(
    \begin{array}{*{15}c}
        * & \cdots & * & 1 & \cdots & 1 & \cellcolor{gray!20} * & \cellcolor{gray!20} \cdots & \cellcolor{gray!20} \cellcolor{gray!20} * &        & \cellcolor{gray!20} * & \cellcolor{gray!20} \cdots & \cellcolor{gray!20} * \\
        * & \cdots & * & 0 & \cdots & 0 & 1 & \cdots & 1 & \cdots & \cellcolor{gray!20} * & \cellcolor{gray!20} \cdots & \cellcolor{gray!20} *\\ 
        * & \cdots & * & 0 & \cdots & 0 & 0 & \cdots & 0 & \cdots & \cellcolor{gray!20} \vdots & \cellcolor{gray!20} \cdots & \cellcolor{gray!20} \vdots \\ 
          & \ddots &   &   &        &   &   & \ddots &   &        & \cellcolor{gray!20} * & \cellcolor{gray!20} \cdots & \cellcolor{gray!20} *\\ 
        * & \cdots & * & 0 & \cdots & 0 & 0 & \cdots & 0 & \cdots & 1 & \cdots & 1\\      
    \end{array}
    \right)
\]
The remaining weights are determined by the powers of~$\theta_k$ that occur in the struts, and so the scaffolding determines the entire weight matrix~\eqref{eq:weight_matrix}. The Laurent polynomial~$f$ is the restriction of~\eqref{eq:pull_back_lift}:
\[
    f = c + \sum_{j \in S_0} x^{\rho_j} \prod_{m=1}^c \theta_m^{-{w_{mj}}}\Big|_{\text{$x_j = 1$ for~$j \in F$}}
\]
But before restriction, each term in~\eqref{eq:pull_back_lift} is homogeneous of degree zero, and we have already reconstructed the matrix of weights. Thus 
\[
    {\tilde{\phi}}^* (W) = c + \sum_{j \in S_0} x^{\rho_j} \prod_{m=1}^c \theta_m^{-{w_{mj}}}
\]
can be uniquely reconstructed from~$f$ by homogeneity. This determines the rays~$\rho_j$,~$j \in S_0$, of the fan for~$Y$; since the other rays are the standard basis for~$N \cong \ZZ^{B'}$ and the toric variety~$Y$ is Fano by assumption, this completely determines~$Y$. Furthermore the weight matrix~\eqref{eq:weight_matrix} determines the subsets~$S_1,\ldots,S_c$, and hence the line bundles~$L_1,\ldots,L_c$. Thus we can reconstruct the presentation of~$X$ as a toric complete intersection~$X \subset Y$ from the scaffolding of~$f$.
\section{Constructions via Laurent inversion}\label{sec:laurent_inversion}
In this section we apply the reconstruction procedure developed in~\S\ref{sec:reconstruction} to several concrete examples.
\subsection{Basic example}\label{sec:basic}
Consider the Laurent polynomial 
\[
  f=x+y+\frac{x}{y}+\frac{2}{y}+\frac{1}{xy}+\frac{2}{x}+\frac{y}{x}
\]
with scaffolding
\[
  f=\textcolor{MidnightBlue}{\frac{(x+y+1)^2}{xy}}+\textcolor{BurntOrange}{(x+y+1)}-3
\]
The Newton polytope~$P$ of~$f$ is as pictured in Figure~\ref{fig:dP5}.
\begin{figure}[ht]
\centering
\begin{tikzpicture}

\draw[black, very thick] (-6,-1) -- (-4,-1) -- (-4,0) -- (-5,1) -- (-6,1) -- cycle;

\filldraw[black] (-6,-1) circle (1.5pt) node[anchor=north east] {$1$};
\filldraw[black] (-5,-1) circle (1.5pt) node[anchor=north] {$2$};
\filldraw[black] (-4,-1) circle (1.5pt) node[anchor=north west] {$1$};
\filldraw[black] (-4,0) circle (1.5pt) node[anchor=west] {$1$};
\filldraw[black] (-5,1) circle (1.5pt) node[anchor=south] {$1$};
\filldraw[black] (-6,1) circle (1.5pt) node[anchor=south east] {$1$};
\filldraw[black] (-6,0) circle (1.5pt) node[anchor=east] {$2$};

\filldraw[color=MidnightBlue!60, fill=MidnightBlue!65, very thick] (-6+1/10,-1+1/10) -- (-4-2/10,-1+1/10) -- (-6+1/10,1-2/10) -- cycle;

\filldraw[color=BurntOrange!60, fill=BurntOrange!65, very thick] (-4-1.5/10,0+1/20) -- (-5+1/20,0+1/20) -- (-5+1/20,1-1.5/10) -- cycle;

\draw[black, thick] (-5.05,0.05) -- (-4.95,-0.05);
\draw[black, thick] (-5.05,-0.05) -- (-4.95,0.05);

\end{tikzpicture}
\caption{A scaffolding of~$f$ with shape~$\PP^2$.}
\label{fig:dP5}
\end{figure}
To reconstruct a toric intersection from the scaffolding, we proceed as in~\S\ref{sec:reconstruction}. We find~$c=1$, and 
\[
  \theta_1(x,y) = 1 + x + y.
\]
Thus~$|S_1| = 3$. Changing the constant term of~$f$ does not affect the Fano manifold that corresponds to~$X$ under mirror symmetry, so we consider
\[
  f = 1 + \textcolor{MidnightBlue}{\frac{(x+y+1)^2}{xy}}+\textcolor{BurntOrange}{(x+y+1)}
\]
This gives~$|S_0| = 2$, and the weight matrix~\eqref{eq:weight_matrix} as 
\[
  \begin{pmatrix}
    -2 & -1 & 1 & 1 & 1
  \end{pmatrix}
\]
In co-ordinates~$(x_1,x_2,x_3,x_4,x_5)$ with~$x_4 = x$ and~$x_5 = y$, we see that~$f$ homogenises to
\[
  1 + 
  \frac{(x_3 + x_4 + x_5)^2}{x_4 x_5} +
  \frac{x_3 + x_4 + x_5}{x_3}
\]
which is~\eqref{eq:pull_back_lift}. Thus the ray map~$\rho$ in~\eqref{eq:ray-sequence-plus-tower-with-splitting} has matrix
\[
  \begin{pmatrix}
     0 & -1 & 1 & 0 & 0 \\
    -1 &  0 & 0 & 1 & 0 \\
    -1 &  0 & 0 & 0 & 1
  \end{pmatrix}
\] 
and the weight matrix~$D$ for the toric variety~$Y$ is
\begin{equation}
  \label{eq:weight-matrix-for-example-1}
  \begin{pmatrix}
    1 & 0 & 0 & 1 & 1 \\
    0 & 1 & 1 & 0 & 0
  \end{pmatrix}
\end{equation}
We see that~$Y$ is~$\PP^2 \times \PP^1$, and that~$X$ is cut out of~$Y$ by a section of~$\sum_{j \in S_1} D_j = \oo(2,1)$.

Fix Cox co-ordinates~$(z_1,z_2,z_3,z_4,z_5)$ on~$Y$ compatible with~\eqref{eq:weight-matrix-for-example-1}, so that~$[z_1:z_4:z_5]$ are projective co-ordinates on~$\PP^2$ and~$[z_2:z_3]$ are projective co-ordinates on~$\PP^1$. Then~$X_P$ is cut out of~$Y$ by the binomial section 
\[
  z_1^2 z_2 = z_3 z_4 z_5
\] 
of~$\oo(2,1)$, by Remark~\ref{rem:binomial}, and this smooths to a general section
\[
  f_2(z_1,z_4,z_5) z_2 + g_2(z_1,z_4,z_5) z_3
\]
of~$\oo(2,1)$, where~$f_2$ and~$g_2$ are polynomials of degree two. Projecting to the first factor~$\PP^2$ of~$Y = \PP^2 \times \PP^1$ exhibits the hypersurface~$X \subset Y$ as the blow-up of~$\PP^2$ in four points.

\subsection{Wedge shapes} \label{sec:wedges}
Another scaffolding of the Laurent polynomial
\[
  f=x+y+\frac{x}{y}+\frac{2}{y}+\frac{1}{xy}+\frac{2}{x}+\frac{y}{x}
\]
from~\S\ref{sec:basic} is
\[
  f=\textcolor{BurntOrange}{\frac{(1+x)^2 + (1+x)y}{x}}+
  \textcolor{MidnightBlue}{\frac{(1+x)^2 + (1+x)y}{xy}}-3
\]
The Newton polytopes of~$f$ and these struts are pictured in Figure~\ref{fig:dP5wedge}.
\begin{figure}[ht]
\centering
\begin{tikzpicture}


\draw[black, very thick] (-6,-1) -- (-4,-1) -- (-4,0) -- (-5,1) -- (-6,1) -- cycle;

\filldraw[black] (-6,-1) circle (1.5pt) node[anchor=north east] {$1$};
\filldraw[black] (-5,-1) circle (1.5pt) node[anchor=north] {$2$};
\filldraw[black] (-4,-1) circle (1.5pt) node[anchor=north west] {$1$};
\filldraw[black] (-4,0) circle (1.5pt) node[anchor=west] {$1$};
\filldraw[black] (-5,1) circle (1.5pt) node[anchor=south] {$1$};
\filldraw[black] (-6,1) circle (1.5pt) node[anchor=south east] {$1$};
\filldraw[black] (-6,0) circle (1.5pt) node[anchor=east] {$2$};

\filldraw[color=MidnightBlue!60, fill=MidnightBlue!65, very thick] (-6+1/10,-1+1/10) -- (-4-2/10,-1+1/10) -- (-5-0.25/10,-1/20) -- (-6+1/10,-1/20) -- cycle;

\filldraw[color=BurntOrange!60, fill=BurntOrange!65, very thick] (-4-2/10,0+1/20) -- (-6+1/10,0+1/20) -- (-6+1/10,1-1.5/10 ) -- (-5-0.25/10,1-1.5/10) -- cycle;

\draw[black, thick] (-5.05,0.05) -- (-4.95,-0.05);
\draw[black, thick] (-5.05,-0.05) -- (-4.95,0.05);

\end{tikzpicture}
\caption{A scaffolding of~$f$ with shape~$\FF_1$.}
\label{fig:dP5wedge}
\end{figure}
Reconstructing toric complete intersection data as in~\S\ref{sec:reconstruction}, we find that~$c=2$, and 
\begin{align*}
  \theta_1(x,y) &= 1+x \\
  \theta_2(x,y) &= (1+x)^2 + (1+x) y
\end{align*}
Thus~$|S_1| = |S_2| = 2$. Changing the constant term of~$f$, as before, we consider
\[
  f = 2 + 
  \textcolor{BurntOrange}{\frac{(1+x)^2 + (1+x)y}{x}}+
  \textcolor{MidnightBlue}{\frac{(1+x)^2 + (1+x)y}{xy}}
\]
This gives~$|S_0| = 2$, and the weight matrix~\eqref{eq:weight_matrix} as 
\[
  \begin{pmatrix}
     0 &  0 & 1 & 1 & -2 & -1 \\
    -1 & -1 & 0 & 0 &  1 &  1
  \end{pmatrix}
\]
In co-ordinates~$(x_1,x_2,x_3,x_4,x_5,x_6)$ with~$x_4 = x$ and~$x_6=y$, we see that~$f$ homogenises to
\[
  2 + 
  \frac{(x_3+x_4)^2 x_5 + (x_3+x_4)x_6}{x_3 x_4 x_5} +
  \frac{(x_3+x_4)^2 x_5 + (x_3+x_4)x_6}{x_4 x_6}
\]
which is~\eqref{eq:pull_back_lift}. The ray map~$\rho$ in~\eqref{eq:ray-sequence-plus-tower-with-splitting} therefore has matrix
\[
  \begin{pmatrix}
    -1 &  0 & 1 & 0 & 0 & 0 \\
    -1 & -1 & 0 & 1 & 0 & 0 \\
    -1 &  0 & 0 & 0 & 1 & 0 \\
     0 & -1 & 0 & 0 & 0 & 1
  \end{pmatrix}
\] 
and the weight matrix~$D$ for the toric variety~$Y$ is
\begin{equation}
  \label{eq:weight-matrix-for-example-2}
  \begin{pmatrix}
    1 & 0 & 1 & 1 & 1 & 0 \\
    0 & 1 & 0 & 1 & 0 & 1
  \end{pmatrix}
\end{equation}
The line bundles here are
\begin{align*}
  L_1 = \sum_{j \in S_1} D_j = \oo(1,2) &&
  L_2 = \sum_{j \in S_2} D_j = \oo(1,1) 
\end{align*}
and~$X$ is cut out of the toric variety~$Y$ by a section of~$L_1 \oplus L_2$.

Note that~$L_2$ coincides with the toric divisor~$D_4$ given by the fourth column of the weight matrix~\eqref{eq:weight-matrix-for-example-2}. A general section~$s_2$ of~$L_2$, in Cox co-ordinates~$(z_1,z_2,z_3,z_4,z_5,z_6)$ on~$Y$ compatible with~\eqref{eq:weight-matrix-for-example-2}, is
\[
  z_1 z_2 + z_1 z_6 + z_2 z_1 + z_2 z_5 + z_3 z_6 + z_4 + z_5 z_6
\]
where we omit general coefficients from the equation. Thus we can solve~$s_2 = 0$ for~$z_4$, eliminating both the fourth column of~\eqref{eq:weight-matrix-for-example-2} and the line bundle~$L_2$, and recovering the weight matrix for~$\PP^1 \times \PP^2$ and the line bundle~$L_1 = \oo(1,2)$ as in the previous example. 

\subsection{Two rigid MMLPs with the same Newton polytope}\label{sec:two_rigid_MMLPs}
Consider the following Laurent polynomials:
\begin{align*}
f_1& =  x + yz + y + \frac{y}{x} + \frac{z}{x} + \frac{2}{x} + \frac{1}{xz} + \frac{z}{xy} + \frac{2}{xy} + \frac{1}{xyz}
\\
f_2& = x + yz + y + \frac{y}{x} + \frac{z}{x} + \frac{3}{x} + \frac{1}{xz} + \frac{z}{xy} + \frac{2}{xy} + \frac{1}{xyz}
\end{align*}
These have the same Newton polytope~$P$, pictured in Figure~\ref{fig:Newton-polytope-for-f1-f2} below; this is a canonical polytope, and is~$\GL(3,\ZZ)$-equivalent to the three-dimensional canonical polytope with ID~\href{http://grdb.co.uk/search/toricf3c?ID_cmp=in&ID=427129}{427129} in the the Graded Ring Database~\cite{GRDB-toric3}.

\begin{figure}[ht]
\centering
\begin{tikzpicture}[every node/.style={scale=0.8}]


\draw[gray!80, thick] (0,2) -- (-1,1/2);
\draw[gray!80, thick] (0,2) -- (3/2,-3/2);
\draw[gray!80, thick] (-1,0.5) -- (-1,-3/2);
\draw[gray!80, thick] (-1,0.5) -- (-5/2,0);
\draw[gray!80, thick] (-1,0.5) -- (3/2,-3/2);
\draw[gray!80, thick] (-5/2,0) -- (-1,-3/2);
\draw[gray!80, thick] (3/2,-3/2) -- (-1,-3/2);
\draw[gray!80, thick] (3/2,-3/2) -- (5/2,-2);
\draw[gray!80, thick] (-3/2,-5/2) -- (-1,-3/2);

\draw[black, very thick] (0,2) -- (-5/2,0);
\draw[black, very thick] (0,2) -- (-0.5,-3);
\draw[black, very thick] (0,2) -- (5/2,-2);
\draw[black, very thick] (-5/2,0) -- (-1/2,-3);
\draw[black, very thick] (-5/2,0) -- (-3/2,-5/2);
\draw[black, very thick] (-3/2,-5/2) -- (-1/2,-3);
\draw[black, very thick] (-1/2,-3) -- (5/2,-2);

\draw[black, thick] (0.05,0.05) -- (-0.05,-0.05);
\draw[black, thick] (0.05,-0.05) -- (-0.05,0.05);

\filldraw[black] (0,2) circle (1.5pt) node[anchor=south] {$(1,0,0)$};
\filldraw[black] (-1,1/2) circle (1pt) node[anchor=south] {$(0,1,0)$};
\filldraw[black] (0,0) circle (0.1pt) node[anchor=west] {$(0,0,0)$};
\filldraw[black] (-5/2,0) circle (1.5pt) node[anchor=east] {$(0,1,1)$};
\filldraw[black] (-1,-3/2) circle (1pt) node[anchor=south] {$(-1,1,0)$};
\filldraw[black] (3/2,-3/2) circle (1pt) node[anchor=south east] {$(-1,0,-1)$};
\filldraw[black] (5/2,-2) circle (1.5pt) node[anchor=north west] {$(-1,-1,-1)$};
\filldraw[black] (-1/2,-3) circle (1.5pt) node[anchor=north west] {$(-1,-1,1)$};
\filldraw[black] (-3/2,-5/2) circle (1.5pt) node[anchor=north east] {$(-1,0,1)$};

\filldraw[black] (0,-2) circle (1pt) node[anchor=south] {};
\filldraw[black] (1,-5/2) circle (1.5pt) node[anchor=south] {};

\end{tikzpicture}
\caption{The Newton polytope~$P$ for~$f_1$ and~$f_2$.}
\label{fig:Newton-polytope-for-f1-f2}
\end{figure}
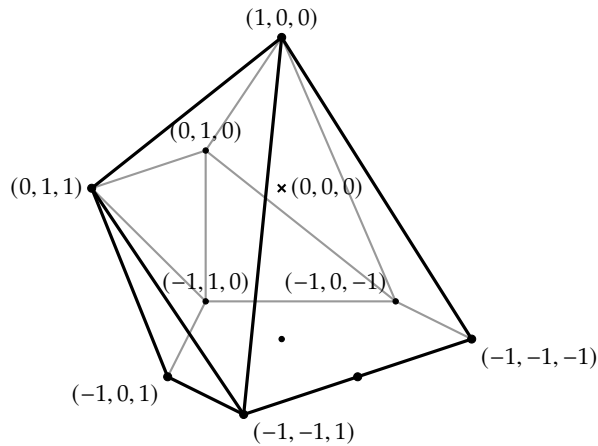

If~$X$ is a \QQFano{}~$3$\nobreakdash-fold that corresponds under mirror symmetry to either~$f_1$ or~$f_2$ then, as discussed, we expect that there is a qG\nobreakdash-degeneration with general fiber~$X$ and special fiber~$X_P$. In particular, therefore, the Hilbert series of~$X$ will coincide with the Hilbert series of~$X_P$. The singularities of any \QQFano{}~$3$\nobreakdash-fold are determined by its Hilbert series~\cite{AltinokBrownReid2002,BrownKasprzyk2022}, and in this case this suggests that~$X$ should have exactly two singularities, both of type~$\frac{1}{2}(1,1,1)$.

\begin{figure}[ht]
\centering
\begin{tikzpicture}

\filldraw[color=MidnightBlue!10, ultra thick, fill=MidnightBlue!50, ultra thick] (-1-3.5,-3/2) -- (-3.5,-2) -- (3/2-3.5,-3/2) -- cycle;
\filldraw[color=MidnightBlue!65, fill=MidnightBlue!65, ultra thick] (-5/2-3.5,0) -- (-1-3.5,1/2) -- cycle;
\filldraw[color=MidnightBlue!10, ultra thick, fill=MidnightBlue!50, ultra thick] (-0.5-3.5,-3) -- (5/2-3.5,-2) -- (-3.5,-2) -- (-3/2-3.5,-5/2) -- cycle;

\draw[gray!80, thick] (-3.5,2) -- (-1-3.5,1/2);
\draw[gray!80, thick] (-3.5,2) -- (3/2-3.5,-3/2);
\draw[gray!80, thick] (-4.5,0.5) -- (-1-3.5,-3/2);
\draw[gray!80, thin] (-4.5,0.5) -- (-5/2-3.5,0);
\draw[gray!80, thick] (-4.5,0.5) -- (3/2-3.5,-3/2);
\draw[gray!80, thick] (-5/2-3.5,0) -- (-1-3.5,-3/2);
\draw[gray!80, thick] (3/2-3.5,-3/2) -- (-1-3.5,-3/2);
\draw[gray!80, thick] (3/2-3.5,-3/2) -- (5/2-3.5,-2);
\draw[gray!80, thick] (-3/2-3.5,-5/2) -- (-1-3.5,-3/2);

\draw[black, very thick] (-5/2-3.5,0) -- (-1/2-3.5,-3);
\draw[black, very thick] (-5/2-3.5,0) -- (-3/2-3.5,-5/2);
\draw[black, very thick] (-3/2-3.5,-5/2) -- (-1/2-3.5,-3);
\draw[black, very thick] (-1/2-3.5,-3) -- (5/2-3.5,-2);
\draw[black, very thick] (-3.5,2) -- (-5/2-3.5,0);
\draw[black, very thick] (-3.5,2) -- (-0.5-3.5,-3);
\draw[black, very thick] (-3.5,2) -- (5/2-3.5,-2);

\draw[black, thick] (0.05-3.5,0.05) -- (-0.05-3.5,-0.05);
\draw[black, thick] (0.05-3.5,-0.05) -- (-0.05-3.5,0.05);

\filldraw[black] (0-3.5,2) circle (1.5pt) node[anchor=south west] {$1$};
\filldraw[black] (-1-3.5,1/2) circle (1pt) node[anchor=south] {$1$};
\filldraw[black] (-3.5,0) circle (0.1pt) node[anchor=west] {$0$};
\filldraw[black] (-5/2-3.5,0) circle (1.5pt) node[anchor=east] {$1$};
\filldraw[black] (-1-3.5,-3/2) circle (1pt) node[anchor=south west] {$1$};
\filldraw[black] (3/2-3.5,-3/2) circle (1pt) node[anchor=south west] {$1$};
\filldraw[black] (5/2-3.5,-2) circle (1.5pt) node[anchor=north west] {$1$};
\filldraw[black] (-1/2-3.5,-3) circle (1.5pt) node[anchor=north west] {$1$};
\filldraw[black] (-3/2-3.5,-5/2) circle (1.5pt) node[anchor=north east] {$1$};
\filldraw[black] (0-3.5,-2) circle (1pt) node[anchor=north west] {$2$};
\filldraw[black] (1-3.5,-5/2) circle (1.5pt) node[anchor=north west] {$2$};
\filldraw[color=MidnightBlue!65, fill=Black, ultra thick] (-3.5,2) circle (2.75pt);


\draw[color=BurntOrange, ultra thick] (-1+3.5,-3/2) -- (3.5,-2);
\draw[color=BurntOrange, ultra thick] (-5/2+3.5,0) -- (-1+3.5,1/2);
\filldraw[color=BurntOrange!15, ultra thick, fill=BurntOrange!70, ultra thick] (-0.5+3.5,-3) -- (5/2+3.5,-2) -- (+3.5+3/2,-3/2) -- (-3/2+3.5,-5/2) -- cycle;

\draw[gray!80, thick] (+3.5,2) -- (-1+3.5,1/2);
\draw[gray!80, thick] (+3.5,2) -- (3/2+3.5,-3/2);
\draw[gray!80, thick] (2.5,0.5) -- (-1+3.5,-3/2);
\draw[gray!80, thin] (2.5,0.5) -- (-5/2+3.5,0);
\draw[gray!80, thick] (2.5,0.5) -- (3/2+3.5,-3/2);
\draw[gray!80, thick] (-5/2+3.5,0) -- (-1+3.5,-3/2);
\draw[gray!80, thick] (3/2+3.5,-3/2) -- (-1+3.5,-3/2);
\draw[gray!80, thick] (3/2+3.5,-3/2) -- (5/2+3.5,-2);
\draw[gray!80, thick] (-3/2+3.5,-5/2) -- (-1+3.5,-3/2);

\draw[black, very thick] (+3.5,2) -- (-5/2+3.5,0);
\draw[black, very thick] (+3.5,2) -- (-0.5+3.5,-3);
\draw[black, very thick] (+3.5,2) -- (5/2+3.5,-2);
\draw[black, very thick] (-5/2+3.5,0) -- (-1/2+3.5,-3);
\draw[black, very thick] (-5/2+3.5,0) -- (-3/2+3.5,-5/2);
\draw[black, very thick] (-3/2+3.5,-5/2) -- (-1/2+3.5,-3);
\draw[black, very thick] (-1/2+3.5,-3) -- (5/2+3.5,-2);

\draw[black, thick] (0.05+3.5,0.05) -- (-0.05+3.5,-0.05);
\draw[black, thick] (0.05+3.5,-0.05) -- (-0.05+3.5,0.05);

\filldraw[black] (+3.5,2) circle (1.5pt) node[anchor=south west] {$1$};
\filldraw[black] (-1+3.5,1/2) circle (1pt) node[anchor=south] {$1$};
\filldraw[black] (+3.5,0) circle (0.1pt) node[anchor=west] {$0$};
\filldraw[black] (-5/2+3.5,0) circle (1.5pt) node[anchor=east] {$1$};
\filldraw[black] (-1+3.5,-3/2) circle (1pt) node[anchor=south west] {$1$};
\filldraw[black] (3/2+3.5,-3/2) circle (1pt) node[anchor=south west] {$1$};
\filldraw[black] (5/2+3.5,-2) circle (1.5pt) node[anchor=north west] {$1$};
\filldraw[black] (-1/2+3.5,-3) circle (1.5pt) node[anchor=north west] {$1$};
\filldraw[black] (-3/2+3.5,-5/2) circle (1.5pt) node[anchor=north east] {$1$};

\filldraw[black] (0+3.5,-2) circle (1pt) node[anchor=north west] {$3$};
\filldraw[black] (1+3.5,-5/2) circle (1.5pt) node[anchor=north west] {$2$};

\filldraw[color=BurntOrange!65, fill=Black, ultra thick] (3.5,2) circle (2.75pt);

\end{tikzpicture}
\caption{Scaffoldings of~$f_1$ and~$f_2$.}
\label{fig:scaffoldings}

\end{figure}
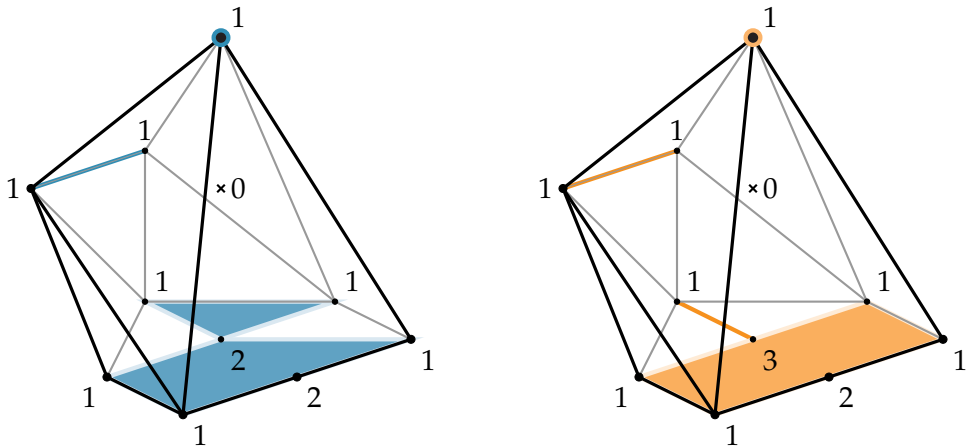

Consider the scaffoldings of~$f_1$ and~$f_2$ shown in Figure~\ref{fig:scaffoldings}; the scaffolding of~$f_1$ has shape~$\FF_1$ and the scaffolding of~$f_2$ has shape~$\PP^1 \times \PP^1$. In each case~$|S_0 \setminus B| = 1$ and the vertex of~$P$ corresponding to the element of~$S_0 \setminus B$ is indicated with a circle. As we will see, applying Laurent inversion to these scaffoldings produces toric complete intersections~$X_1 \subset Y_1$ and~$X_2 \subset Y_2$ such that~$Y_1$ and~$Y_2$ are toric orbifolds. Each of the varieties~$X_1$ and~$X_2$ is a \QQFano{} $3$\nobreakdash-fold with singular locus consisting of exactly two singularities of type~$\frac{1}{2}(1,1,1)$. The varieties~$X_1$ and~$X_2$ are not isomorphic (or even deformation equivalent) to each other, because they have distinct quantum periods.

\subsubsection{The scaffolding of~$f_1$}
Here
\[
  f_1 = (1+z^{-1}) y z + \frac{(1+z^{-1})(1+z^{-1}+y)z}{xy} + \frac{1+z^{-1}+y}{x} + x
\]
which, after the change of variables~$z \mapsto z^{-1}$, gives
\[
  \frac{(1+z) y}{z} + \frac{(1+z)(1+z+y)}{xyz} + \frac{1+z+y}{x} + x
\]
Following~\S\ref{sec:reconstruction} again, we find that~$c=2$ and 
\begin{align*}
  \theta_1(x,y,z) &= 1 + z \\
  \theta_2(x,y,z) &= 1 + z + y
\end{align*}
Thus~$|S_1| = |S_2| = 2$. Shifting the constant term, as before, we consider
\[
  2 + \frac{(1+z) y}{z} + \frac{(1+z)(1+z+y)}{xyz} + \frac{1+z+y}{x} + x
\]
This gives~$|S_0| = 4$, and the weight matrix~\eqref{eq:weight_matrix} as 
\[
  \begin{pmatrix}
    -1 & -1 &  0 & 0 & 1 & 1 & -1 & 0 \\
     0 & -1 & -1 & 0 & 0 & 0 &  1 & 1
  \end{pmatrix}
\]
In co-ordinates~$(x_1,x_2,\ldots,x_8)$ with~$x_4 = x$,~$x_6 = z$, and~$x_8=y$ we see that~$f$ homogenises to
\[
  2 + 
  \frac{(x_5+x_6) x_8}{x_5 x_6 x_7} + 
  \frac{(x_5+x_6)(x_5 x_7 + x_6 x_7 + x_8)}{x_4 x_6 x_8} + 
  \frac{x_5 x_7 + x_6 x_7 + x_8}{x_4 x_5 x_7} + 
  x_4
\]
which is~\eqref{eq:pull_back_lift}. Thus the ray map~$\rho$ in~\eqref{eq:ray-sequence-plus-tower-with-splitting} has matrix
\begin{equation}
  \label{eq:ray-matrix-for-example-3-1}  
  \begin{pmatrix}
    0 & -1 & -1 & 1 & 0 & 0 & 0 & 0 \\ 
    -1 &  0 & -1 & 0 & 1 & 0 & 0 & 0 \\ 
    -1 & -1 &  0 & 0 & 0 & 1 & 0 & 0 \\ 
    -1 &  0 & -1 & 0 & 0 & 0 & 1 & 0 \\
    1 & -1 &  0 & 0 & 0 & 0 & 0 & 1 
  \end{pmatrix}
\end{equation}
and the weight matrix~$D$ for the toric variety~$Y$ is
\begin{equation}
  \label{eq:weight-matrix-for-example-3-1}
  \begin{pmatrix}
    1 & 0 & 0 & 0 & 1 & 1 & 1 & -1 \\
    0 & 1 & 0 & 1 & 0 & 1 & 0 &  1 \\
    0 & 0 & 1 & 1 & 1 & 0 & 1 &  0
  \end{pmatrix}
\end{equation}
The line bundles here are 
\begin{align*}
  L = \sum_{j \in S_1} D_j = \oo(2,1,1) &&
  L' = \sum_{j \in S_2} D_j = \oo(0,1,1) 
\end{align*}
and our analysis suggests that we should consider a Fano variety~$X_1$ cut out of a toric variety~$Y_1$ with ray map~\eqref{eq:ray-matrix-for-example-3-1} by a section of~$L \oplus L'$. Since~$L'$ occurs as the toric divisor given by the fourth column of~\eqref{eq:weight-matrix-for-example-3-1}, we remove the fourth column and also~$L'$, considering instead the toric variety~$Y_1$ obtained as a GIT quotient~$\CC^7 \sslash (\Cstar)^3$ where the weight matrix for the action is
\begin{equation}
  \label{eq:smaller-weight-matrix-for-example-3-1}
  \begin{pmatrix}
    1 & 0 & 0 & 1 & 1 & 1 & -1 \\
    0 & 1 & 0 & 0 & 1 & 0 &  1 \\
    0 & 0 & 1 & 1 & 0 & 1 &  0
  \end{pmatrix}
\end{equation}
and the subvariety~$X_1 \subset Y_1$ cut out by a section of~$L \cong \oo(2,1,1)$.

To specify~$Y_1$, we need to choose a stability condition for the GIT quotient~$Y_1 = \CC^7 \sslash (\Cstar)^3$; equivalently, we need to choose the fan for~$Y$, and the weight matrix~\eqref{eq:smaller-weight-matrix-for-example-3-1} only determines the rays of this fan. For this we examine the secondary fan. The secondary fan is the cone spanned by the columns of~\eqref{eq:smaller-weight-matrix-for-example-3-1}, which we picture by intersecting it with the plane~$x+2y+z=1$ and projecting to the~$xy$-plane, equipped with the wall-and-chamber decomposition shown in Figure~\ref{fig:secondary-fan-example-3-1}. Choosing the stability condition that makes~$Y_1$ into a Fano toric variety, i.e.~${-K}_{Y_1} = \oo(3,3,3)$, which is shown as a hollow circle in Figure~\ref{fig:secondary-fan-example-3-1}, results in a non-$\QQ$-factorial toric variety. We instead choose the stability condition~${-K}_{Y_2} - L = \oo(1,2,2)$, which lies in the interior of the shaded chamber. This specifies a toric orbifold~$Y_1$, and ensures that the subvariety~$X_1 \subset Y_1$ cut out by a general section of~$L$ is Fano.

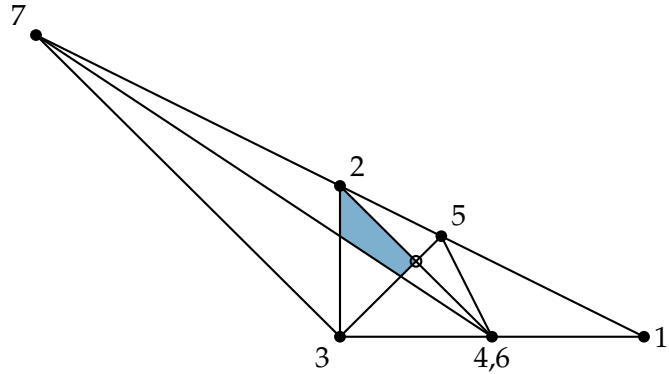
\begin{figure}[th]
\centering
\begin{tikzpicture}[scale=4.0]

\filldraw[MidnightBlue!40] (0,1/3) -- (0,1/2) -- (1/4,1/4) -- (1/5,1/5) -- cycle;

\draw[color=black, thick] (-1,1) -- (1,0) -- (0,0)--cycle;
\draw[color=black, thick] (-1,1) -- (1/2,0);
\draw[color=black, thick] (0,1/2) -- (1/2,0);
\draw[color=black, thick] (1/3,1/3) -- (1/2,0);
\draw[color=black, thick] (0,0) -- (0,1/2);
\draw[color=black, thick] (0,0) -- (1/3,1/3);

\filldraw (1,0) circle (0.5pt) node[anchor=west] {$1$};
\filldraw (0,1/2) circle (0.5pt) node[anchor=south west] {$2$};
\filldraw (0,0) circle (0.5pt) node[anchor=north east] {$3$};
\filldraw (1/2,0) circle (0.5pt) node[anchor=north] {$4$,$6$};
\filldraw (1/3,1/3) circle (0.5pt) node[anchor=south west] {$5$};
\filldraw (-1,1) circle (0.5pt) node[anchor=south east] {$7$};

\draw[color=black, thick] (1/4,1/4) circle (0.5pt) node {};
\end{tikzpicture}
\caption{A slice of the secondary fan of~$Y_1$. The shaded region is the stability chamber.}
\label{fig:secondary-fan-example-3-1}
\end{figure}

It remains to determine the singularities of a general section of~$L$. Fix Cox co-ordinates $(z_1,z_2,\ldots,z_7)$ on~$Y_2$ compatible with~\eqref{eq:smaller-weight-matrix-for-example-3-1}. There are precisely twelve maximal charts on~$Y_1$, each of the form~$U_{ijk}\coloneqq\{z_i=z_j=z_k=1\}$ where the cone~$\langle ijk \rangle$ contains the shaded chamber in its strict interior. Only three of them are singular:~$U_{357}$,~$U_{457}$ and~$U_{567}$. A general section of~$L$ is of the form
 \[  
  z_1^3 z_3 z_7 + z_4 z_5 + z_1 z_2 z_4 + z_5 z_6 + z_1^2 z_4 z_7 + z_1 z_3 z_5 + z_1 z_2 z_6 + z_1^2 z_2 z_3 + z_1^2 z_6 z_7
\] 
where as usual we omit generic coefficients from the equation. We see that:
\begin{itemize}
  \item The singular locus of~$Y_1$ consists of the origins of the three singular charts and the curve~$C=\{z_1=z_2=z_3=0\}$. 
  \item A general section of~$L$ does not pass through the origins of~$U_{457}$ and~$U_{567}$, but does pass through the origin of~$U_{357}$. This gives rise to a singularity on~$X_1$ of type~$\frac{1}{2}(1,1,1)$: the chart on~$Y_1$ is~$\frac{1}{2}(1,1,1,1)_{1246}$ and~$L=\oo(1)$ here.
  \item A general section of~$L$ meets~$C$ in an isolated point of type~$\frac{1}{2}(1,1,1)$.
\end{itemize}
Thus a general section of~$L$ is singular in precisely two points, and each is of type~$\frac{1}{2}(1,1,1)$.

\subsubsection{The scaffolding of~$f_2$}
This is
\[
  f_2 = (1+z) y + \frac{(1+z)^2(1+y)}{xyz} + \frac{1+y}{x} + x
\]
which gives~$c=2$ and 
\begin{align*}
  \theta_1(x,y,z) &= 1 + z \\
  \theta_2(x,y,z) &= 1 + y
\end{align*}
Thus~$|S_1| = |S_2| = 2$. Shifting the constant term again, we consider
\[
  2 + (1+z) y + \frac{(1+z)^2(1+y)}{xyz} + \frac{1+y}{x} + x
\]
This gives~$|S_0| = 4$, and the weight matrix~\eqref{eq:weight_matrix} as 
\[
  \begin{pmatrix}
    -1 & -2 &  0 & 0 & 1 & 1 & 0 & 0 \\
     0 & -1 & -1 & 0 & 0 & 0 & 1 & 1
  \end{pmatrix}
\]
In co-ordinates~$(x_1,x_2,\ldots,x_8)$ with~$x_4 = x$,~$x_6 = z$, and~$x_8=y$ we see that~$f$ homogenises to
\[
  2 + 
  \frac{(x_5+x_6) x_8}{x_5 x_7} + 
  \frac{(x_5+x_6)^2(x_7+x_8)}{x_4 x_5 x_6 x_8} + 
  \frac{x_7+x_8}{x_4 x_7} + 
  x_4
\]
which is~\eqref{eq:pull_back_lift}. Thus the ray map~$\rho$ in~\eqref{eq:ray-sequence-plus-tower-with-splitting} has matrix
\[
  \begin{pmatrix}
     0 & -1 & -1 & 1 & 0 & 0 & 0 & 0 \\ 
    -1 & -1 &  0 & 0 & 1 & 0 & 0 & 0 \\ 
     0 & -1 &  0 & 0 & 0 & 1 & 0 & 0 \\ 
    -1 &  0 & -1 & 0 & 0 & 0 & 1 & 0 \\
     1 & -1 &  0 & 0 & 0 & 0 & 0 & 1 
  \end{pmatrix}
\]
and the weight matrix~$D$ for the toric variety~$Y$ is
\begin{equation}
  \label{eq:weight-matrix-for-example-3-2}
  \begin{pmatrix}
    1 & 0 & 0 & 0 & 1 & 0 & 1 & -1 \\
    0 & 1 & 0 & 1 & 1 & 1 & 0 &  1 \\
    0 & 0 & 1 & 1 & 0 & 0 & 1 &  0
  \end{pmatrix}
\end{equation}
The line bundles are 
\begin{align*}
  L = \sum_{j \in S_1} D_j = \oo(1,2,0) &&
  L' = \sum_{j \in S_2} D_j = \oo(0,1,1) 
\end{align*}
Once again we remove both the fourth column of~\eqref{eq:weight-matrix-for-example-3-2} and~$L'$, and consider the toric variety~$Y_2$ obtained as a GIT quotient~$\CC^7 \sslash (\Cstar)^3$ where the weight matrix for the action is
\begin{equation}
  \label{eq:smaller-weight-matrix-for-example-3-2}
  \begin{pmatrix}
    1 & 0 & 0 & 1 & 0 & 1 & -1 \\
    0 & 1 & 0 & 1 & 1 & 0 &  1 \\
    0 & 0 & 1 & 0 & 0 & 1 &  0
  \end{pmatrix}
\end{equation}
and the subvariety~$X_2 \subset Y_2$ cut out by a section of~$L \cong \oo(1,2,0)$.

To specify the stability condition for the GIT quotient~$Y_2 = \CC^7 \sslash (\Cstar)^3$, we examine the secondary fan. This is the cone spanned by the columns of~\eqref{eq:smaller-weight-matrix-for-example-3-2}, which we again picture by intersecting it with the plane~$x+2y+z=1$ and projecting to the~$xy$-plane, equipped with the wall-and-chamber decomposition shown in Figure~\ref{fig:secondary-fan-example-3-2}. Choosing the stability condition that makes~$Y_2$ into a Fano toric variety, i.e.~${-K}_{Y_2} = \oo(2,4,2)$, which is shown as a hollow circle in Figure~\ref{fig:secondary-fan-example-3-2}, again results in a toric variety that is not~$\QQ$-factorial. We instead choose the stability condition~${-K}_{Y_1} - L = \oo(1,2,2)$, which lies in the interior of the shaded chamber. This specifies a toric orbifold~$Y_2$, and ensures that the subvariety~$X_2 \subset Y_2$ cut out by a general section of~$L$ is Fano.

\begin{figure}[th]
\centering
\begin{tikzpicture}[scale=4.0]

\filldraw[BurntOrange!45] (0,1/3) -- (0,1/2) -- (1/4,1/4) -- (1/5,1/5) -- cycle;

\draw[color=black, thick] (-1,1) -- (1,0) -- (0,0)--cycle;
\draw[color=black, thick] (-1,1) -- (1/2,0);
\draw[color=black, thick] (0,1/2) -- (1/2,0);
\draw[color=black, thick] (1/3,1/3) -- (1/2,0);
\draw[color=black, thick] (0,0) -- (0,1/2);
\draw[color=black, thick] (0,0) -- (1/3,1/3);

\filldraw (1,0) circle (0.5pt) node[anchor=west] {$1$};
\filldraw (0,1/2) circle (0.5pt) node[anchor=south west] {$2$,$5$};
\filldraw (0,0) circle (0.5pt) node[anchor=north east] {$3$};
\filldraw (1/3,1/3) circle (0.5pt) node[anchor=south west] {$4$};
\filldraw (1/2,0) circle (0.5pt) node[anchor=north] {$6$};
\filldraw (-1,1) circle (0.5pt) node[anchor=south east] {$7$};

\draw[color=black, thick] (1/6,1/3) circle (0.5pt) node {};
\end{tikzpicture}
\caption{A slice of the secondary fan of~$Y_2$. The shaded region is the stability chamber.}
\label{fig:secondary-fan-example-3-2}
\end{figure}
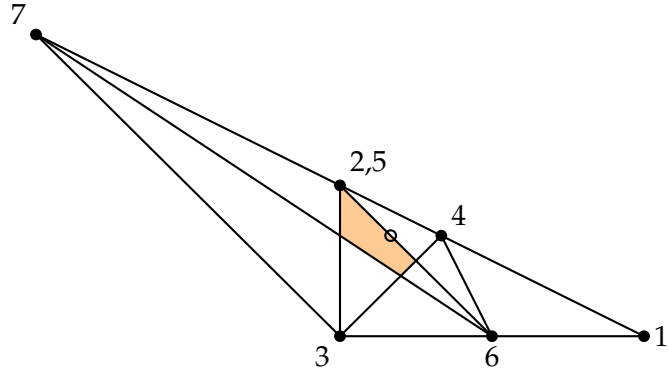

Once again we analyse the singularities of a general section of~$L$, working in Cox co-ordinates~$(z_1,z_2,\ldots,z_7)$ compatible with~\eqref{eq:smaller-weight-matrix-for-example-3-2}. Again there are exactly twelve maximal charts on~$Y_2$, each of the form~$U_{ijk}\coloneqq\{z_i=z_j=z_k=1\}$ where the cone~$\langle ijk \rangle$ contains the shaded chamber in its strict interior. Only two of them are singular:~$U_{347}$ and~$U_{467}$. A general section of~$L$ is of the form
 \[  
  z_1 z_5^2 + z_4 z_5 + z_1^3 z_7^2 + z_1^2 z_2 z_7 + z_1^2 z_5 z_7 + z_1 z_2^2 + z_1 z_4 z_7 + z_1 z_2 z_5 + z_2 z_4
\]
and hence passes through the origins of both singular charts. This gives rise to two singularities on~$X_2$ of type~$\frac{1}{2}(1,1,1)$: the charts on~$Y_2$ are~$\frac{1}{2}(1,1,1,1)_{1256}$ and~$\frac{1}{2}(1,1,1,1)_{1235}$, and in each chart~$L = \oo(1)$. Thus again a general section of~$L$ is singular in precisely two points, each of type~$\frac{1}{2}(1,1,1)$.

\subsubsection*{The varieties~$X_1$ and~$X_2$ are distinct}
To compute the regularised quantum periods of~$X_1$ and~$X_2$, we argue as in~\cite[Corollary~D.5]{CoatesCortiGalkinKasprzyk2016}, but using the mirror theorem for toric complete intersections due to J.~Wang~\cite{Wang2019} in place of the Quantum Lefschetz theorem and Givental's mirror theorem for toric manifolds. This yields
\begin{align*}  
  \widehat{G}_{X_1}(t) & = 1 + 4 t^2 + 18 t^3 + 60 t^4 + 600 t^5 + 2470 t^6 + 18900 t^7 + 118300 t^8 + 723240 t^9 + \cdots \\
  \widehat{G}_{X_2}(t) & = 1 + 6 t^2 + 18 t^3 + 90 t^4 + 780 t^5 + 3210 t^6 + 28560 t^7 + 164010 t^8 + 1146600 t^9 + \cdots
\end{align*}
Thus~$X_1$ and~$X_2$ are not deformation equivalent. Note also that, as the singularities of~$X_1$ and~$X_2$ are isolated, they are rigid~\cite{Schlessinger1968}, and therefore we cannot smooth~$X_1$ or~$X_2$ (or~$X_P$) further in their deformation-equivalence classes.
\section{Systematic generation of hypersurface examples}\label{sec:random_ci}
One can explore the landscape of Fano manifolds by systematically generating complete intersection models~\cite{CoatesKasprzykPrince2015}, but this approach is much less effective in the \QQFano{} (orbifold) setting. That is because, unless the ambient space is a weighted projective space~\cite{IanoFletcher2000}, we lack combinatorial criteria to detect whether a complete intersection is quasismooth, and thus checking quasismoothness involves computationally expensive Gr\"obner basis calculations. It turns out, however, that we can restore the effectiveness of this method in the \QQFano{} setting by combining the systematic generation of toric complete intersections with an analysis of their Laurent polynomial mirrors. By restricting attention to those toric complete intersections such that the Givental/Hori--Vafa mirror is a maximally mutable Laurent polynomial, one can sidestep many expensive quasismoothness checks that we expect, in general, will fail. In this section we use this approach to generate \QQFano{} threefolds that are hypersurfaces in toric orbifolds. One should regard this as a proof of concept: the methods apply without significant change to complete intersections in higher-dimensional toric varieties as well.

We randomly generated \QQFano{} threefolds that occur as toric hypersurfaces, by taking the following steps.
\begin{enumerate}
    \item \label{item:step_1} We generated~$2 \times 6$ integer matrices~$W$ and length-$2$ integer vectors~$D$ with small non-negative integer entries. Specifically, we chose entries in~$W$ and~$D$ uniformly at random from the set~$\{0,1,\ldots,6\}$.
    \item \label{item:step_2} We discarded~$(W,D)$ unless:
    \begin{enumerate}
        \item  the four-dimensional Fano toric variety~$Y$ with weight matrix~$W$ was~$\QQ$-factorial;
        \item the line bundle~$L \to Y$ defined by the weight vector~$D$ was nef; and 
        \item\label{item:step_2_c} ${-K}_Y - L$ was ample on~$Y$.
    \end{enumerate}
    These are combinatorial conditions on the entries of~$W$ and~$L$. Condition~\eqref{item:step_2_c} here guarantees that the threefold cut out by a general section of~$L$ is Fano.
    \item \label{item:step_3} We discarded~$(W,D)$ unless:
    \begin{enumerate}
        \item $D$ did not occur as a column of~$W$.
        \item the \QQFano{} threefold~$X$ cut out by a general section of~$L$ had a truncated period sequence~$c_0,c_1,\ldots, c_N$ such that
        \[
            \gcd \{ d : c_d \ne 0 \} = 1
        \]
        That is, it did not have the pattern of zeroes characteristic of Fano varieties with Fano index greater than one.
        \item\label{item:subclause_c} the weight matrix and the divisor satisfied the conditions for the Givental/Hori--Vafa method to give a Laurent polynomial~$f$ mirror to~$X$~\cite[\S5]{CoatesKasprzykPrince2015}.
        \item\label{item:subclause_d} the Hilbert series of~$X$ was present in the database of possible Hilbert series of semistable \QQFano{} threefolds~\cite{GRDB-fano3}. Note that such a Hilbert series uniquely determines a set of singularities, called the basket, such that any \QQFano{} threefold with that Hilbert series has singular set equal to the basket~\cite{AltinokBrownReid2002,BrownKasprzyk2022}.
        \item the weight matrix~$W$ satisfied certain divisibility conditions that are necessary if the toric variety~$Y$ is to contain a quasismooth hypersurface with singular set equal to the basket from~\eqref{item:subclause_d}.
    \end{enumerate}
    \item \label{item:step_4} We discarded~$(W,D)$ unless the Laurent polynomial mirror~$f$ from~\eqref{item:subclause_c} was rigid maximally mutable.
    \item \label{item:step_5} We discarded~$(W,D)$ unless the hypersurface~$X$ cut out by a randomly chosen section of~$L$ was quasismooth with isolated singularities.
\end{enumerate}
From just under~$5\,000\,000$ examples after step~\ref{item:step_1}, we found~$160\,762$ examples after step~\ref{item:step_2}, then~$7\,272$ examples after step~\ref{item:step_3}, then~$354$ examples after step~\ref{item:step_4}, and~$333$ examples after step~\ref{item:step_5}. 

Note that a single \QQFano{} threefold can correspond to many different Laurent polynomials. But conjecturally these Laurent polynomials are all related by mutation, and in particular have the same period sequence; indeed we expect that a \QQFano{} threefold is uniquely determined by its period sequence. The~$333$ Laurent polynomials above gave rise to~$130$ distinct period sequences. Of these, 32 were not among the period sequences of rigid MMLPs with canonical Newton polytope (as pictured in Figure~\ref{fig:grdb_heatmaps_b}), and~$15$ occurred among the examples constructed using Laurent inversion in~\cite{Heuberger2022}. The position of the~$130$ period sequences in the landscape of \QQFano{} threefolds is indicated in Figure~\ref{fig:grdb_scattergun_zoom}; see also Figure~\ref{fig:grdb_heatmaps_e} above.

\begin{figure}[htbp]
    \centering
    \includegraphics[width=0.7\textwidth]{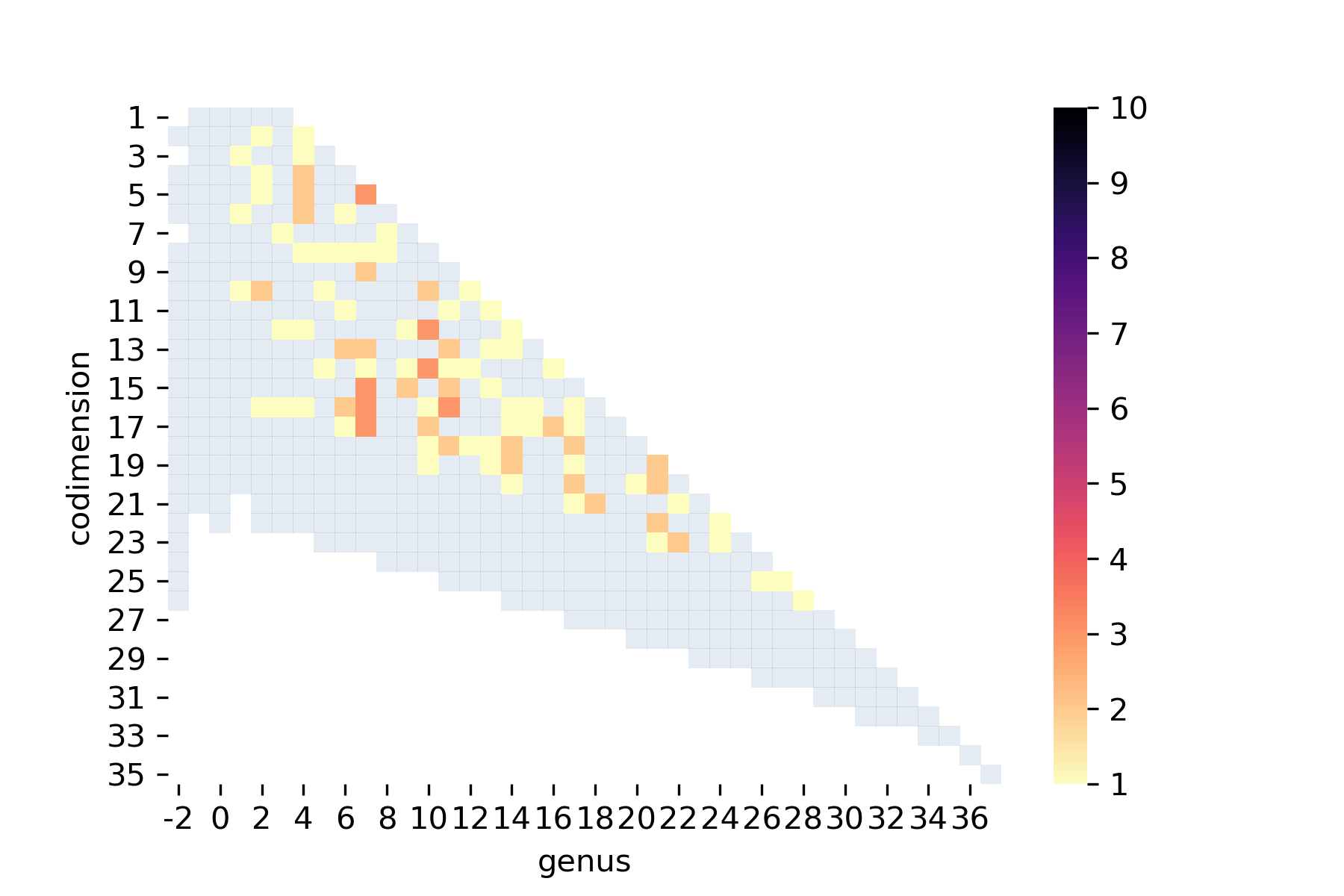}
    \caption{The distribution of Hilbert series for randomly generated {\protect \QQFano} quasismooth threefold hypersurfaces in {\protect $\mathbb{Q}$-factorial} toric varieties of Picard rank~$2$. Hilbert series are recorded as pairs~$(c, g)$ where~$c$ is the estimated codimension and~$g$ is the genus.}
    \label{fig:grdb_scattergun_zoom}
\end{figure}

\section{Towards a Classification Theorem}\label{sec:future}
If Conjecture~\ref{conj:mirror} holds, along with the surrounding conjectural picture discussed in~\S\ref{sec:introduction}, then the classification of \QQFano{} threefolds would follow from understanding:
\begin{enumerate}
    \item[(1)] \hypertarget{item:towards_1}{} Given a rigid MMLP, how can we construct the corresponding Fano variety?
    \item[(2)] \hypertarget{item:towards_2}{}How can we find a representative rigid MMLP for every \QQFano{} threefold?
\end{enumerate}
Let us make the latter question more precise:
\begin{itemize}
    \item[(2$'$)] \hypertarget{item:towards_2_prime}{} Given a finite set of deformation-equivalence classes of \QQFano{} threefolds, how can we find a representative rigid MMLP for every Fano variety in this set?
\end{itemize}
One natural way to create the finite sets in (\hyperlink{item:towards_2_prime}{2$'$}) is by bounding the complexity of the singularities allowed~\cite{Birkar2021}.
\begin{itemize}
    \item[(2$''$)] \hypertarget{item:towards_2_double_prime}{} Given a bound on the complexity of the singularities of a \QQFano{} threefold, how can we find a representative rigid MMLP for every Fano variety that satisfies this bound?
\end{itemize}

\subsection{Towards answering question 1}
Laurent inversion is a powerful tool for addressing question \hyperlink{item:towards_1}{1}. But it is clearly not enough, because not every \QQFano{} threefold is a toric complete intersection. There has been some progress in constructing \QQFano{} threefolds in Pfaffian format from scaffoldings with shapes based on the toric surface~$\mathrm{dP}_7$~\cite{Heuberger2022}. 
\begin{description}
    \item[Problem A] Generalise Laurent inversion to Fano varieties presented in Pfaffian format. 
\end{description}
Furthermore every smooth Fano threefold is a~\emph{quiver flag zero locus}, that is, a zero locus of a section of a homogeneous vector bundle over a GIT quotient of a vector space by a product of general linear groups~\cite{CoatesCortiGalkinKasprzyk2016,Kalashnikov2019}. Many toric complete intersections are also quiver flag zero loci, and generalising Laurent inversion to quiver flag zero loci would be an important source of new constructions. 
\begin{description}
    \item[Problem B] Generalise Laurent inversion to quiver flag zero loci. 
\end{description}
Note that recent work of Webb allows the analysis of quiver flag zero loci that are orbifolds~\cite{Webb2021}. Kalashnikov has produced~$99$ rigid MMLPs in four variables that are conjectural mirrors to quiver flag zero loci~\cite{Kalashnikov2019b}; this should be an important source of test cases.

An additional challenge is that, as things stand, applying Laurent inversion requires substantial ingenuity, particularly in the construction of scaffoldings. There are many deformation classes of \QQFano{} threefolds, far more than it would be practical to construct by hand. 
\begin{description}
    \item[Problem C] Develop effective algorithms to automate Laurent inversion. 
\end{description} 
In order to construct the classification, and even to explore it at scale, we will need to use Laurent inversion as part of computer algebra calculations.

One of the most effective tools for constructing \QQFano{} threefolds in low codimension is Tom and Jerry~\cite{BrownKerberReid2012,BrownReidStevens2021}. Relating this technique to the analysis of rigid MMLPs would potentially be very powerful, in particular because constructing Laurent inversion models in low codimension seems to be difficult.
\begin{description}
    \item[Problem D] Understand the relationship between Tom and Jerry, or more generally projection and unprojection, and mirror symmetry.
\end{description}

\subsection{Towards answering question 2} 
In order to approach question~\hyperlink{item:towards_2}{2}, we need to understand how to determine geometric properties of a Fano variety~$X$ from a Laurent polynomial~$f$ that corresponds to~$X$ under mirror symmetry, or from the Newton polytope~$P = \Newt{f}$. In particular, to answer question \hyperlink{item:towards_2_double_prime}{2$''$}, we need to understand how to determine the singularities of~$X$ from~$f$ or from~$P$. This will have two consequences for our search:
\begin{itemize}
    \item given a polytope~$P$, it will help us to predict a \QQFano{} threefold to which~$X_P$ deforms;
    \item given a basket of singularities~$\cB$, it will help us to bound the class of polytopes~$P$ that we need to analyse in order to find rigid MMLP representatives for all \QQFano{} threefolds with that basket. 
\end{itemize}
In both cases the polytope~$P$ occurs as the Newton polytope of a rigid MMLP~$f$ that corresponds to~$X$.

In two dimensions we have good control over question~\hyperlink{item:towards_2_double_prime}{2$''$}, through the notion of~\emph{singularity content}~\cite{AkhtarKasprzyk2014}. A Fano polygon~$P$ determines a collection of singularities~$\cB$, again called the~\emph{basket}, with the property that a general qG partial smoothing of~$X_P$ has singularities given by~$\cB$. The basket~$\cB$ is given combinatorially as follows. Each edge~$E$ of~$P$ lies at some lattice height~$r$ above the origin, and we subdivide~$E$ into a number of line segments of length~$r$, plus at most one line segment of length less than~$r$. Making such a choice for each edge~$E$ defines a fan~$\Sigma$, which gives a crepant partial resolution~$X_\Sigma$ of~$X_P$. The line segments of length equal to their lattice height define~$T$-singularities~\cite{KollarShepherdBarron1988} on~$X_\Sigma$; the remaining singularities on~$X_\Sigma$ are qG\nobreakdash-rigid, and define the basket~$\cB$. It is not clear~\emph{a priori} that the basket~$\cB$ is independent of choices made, but this turns out to be the case. 

As indicated, in two dimensions singularity content plays two roles:
\begin{itemize}
    \item given a polygon~$P$, it determines the singularities on a orbifold del~Pezzo surface to which~$X_P$ deforms;
    \item given a basket of singularities~$\cB$, it determines the class of polygons that we need to analyse in order to find rigid MMLP representatives for all orbifold del~Pezzo surfaces with basket~$\cB$. This is the class of lattice polygons with singularity content~$\cB$, under the equivalence relation given by combinatorial mutation~\cite{AkhtarCoatesGalkinKasprzyk2012}.
\end{itemize}
This approach allows us to classify orbifold TG del~Pezzo surfaces with a given basket~\cite{AkhtarCoatesCortiHeubergerKasprzykOnetoPetracciPrinceTveiten2016,CortiHeuberger2017,KasprzykNillPrince2017,CaveyKutas2017,CaveyPrince2020,Cuzzucoli2020}. Note that singularity content is defined in terms of the polygon~$P$, rather than than a MMLP~$f$ with Newton polygon~$P$. This is consistent with  the fact that in two dimensions there is a unique family of maximally mutable Laurent polynomials with a given Newton polygon~\cite{CoatesKasprzykPittonTveiten2021}. As we argue below, to obtain a notion of singularity content in higher dimensions we expect that it will be essential to work with~$f$ rather than~$P$.

\subsection{Our pictures of the landscape are unsatisfactory}
Recall that Figure~\ref{fig:grdb_heatmaps_b} was produced by analysing a collection of rigid MMLPs~\cite{Zenodo-certain-MMLP}. We believe that this collection of Laurent polynomials contains almost all\footnote{The algorithm that we use has impractically long runtime on several hundred of the~$674\,688$ three-dimensional canonical polytopes.} rigid MMLPs~$f$ in three variables such that~$\Newt{f}$ is canonical, but that was not so important for the discussion in this paper. Indeed the classes of polytopes that we considered when producing the pictures of the \QQFano{} landscape in Figure~\ref{fig:grdb_heatmaps}, although natural from the point of view of combinatorics, are not well-adapted to the \QQFano{} classification problem. 
Ideally we would search over a classification of three-dimensional lattice polytopes with fixed singularity content, up to the equivalence given by combinatorial mutation. Fixing the singularity content here would correspond to bounding the complexity of the singularities in the corresponding \QQFano{} threefolds. In order to do this, however, we would need an appropriate definition of singularity content for three-dimensional polytopes. This does not yet exist, and so for Figure~\ref{fig:grdb_heatmaps} we had to work with the three-dimensional polytope classifications that are available. 

\subsection{Towards singularity content in higher dimensions}
As discussed, singularity content in dimension two is a basket of singularities, which is computed from a polygon~$P$ by a crepant resolution procedure. In higher dimensions the situation is more complicated. For example, in dimension three there are global obstructions to smoothability even when all local obstructions vanish~\cite{Petracci2020}, but this is not the case in two dimensions~\cite{AkhtarCoatesCortiHeubergerKasprzykOnetoPetracciPrinceTveiten2016}. Furthermore, in three dimensions~$X_P$ can admit many qG\nobreakdash-generizations that are not deformation equivalent; this is not the case in two dimensions. Thus any notion of singularity content in higher dimensions must be richer than just a basket of singularities. 
But nonetheless there are hints that a similar crepant partial resolution procedure might produce the basket for~$X$, together with some extra structure (a triangulation), from its Laurent polynomial mirror~$f$. For example, consider the three-dimensional canonical polytope~$P$ shown in Figure~\ref{fig:singularity_content}. This supports two distinct rigid MMLPs
\begin{align*}
    f_1 &= x^4 y^2 z^3 + 2 x^2 y z^2 + 3 x + y + z + \frac{2}{x y z} + \frac{3}{x^2 y^2 z^3} + \frac{1}{x^5 y^4 z^6} \\
    f_2 &= x^4 y^2 z^3 + 2 x^2 y z^2 + 3 x + y + z + \frac{3}{x y z} + \frac{3}{x^2 y^2 z^3} + \frac{1}{x^5 y^4 z^6} 
\end{align*}
which differ only in the coefficient of~$x^{-1} y^{-1} z^{-1}$. The period sequences for~$f_1$ and~$f_2$ are distinct, so we expect that there are two deformation families of \QQFano{} threefold that qG\nobreakdash-degenerate to~$X_P$. One of these families can be constructed by applying Laurent inversion to the scaffolding shown in Figure~\ref{fig:future-scaffolding}, which is mutation-equivalent to~$f_2$; we do not know how to construct the other deformation family.

\begin{figure}[htb]
    \centering
    \resizebox*{0.49\linewidth}{!}{
        \begin{tikzpicture}[scale=0.7]
        
        \draw[black, very thick] (-1,5) -- (1,-5);
        \draw[black, very thick] (-1,5) -- (-5,-2);
        \draw[black, very thick] (-1,5) -- (7,1);
        \draw[black, very thick] (-5,-2) -- (1,-5);
        \draw[black, very thick] (7,1) -- (1,-5);
        \draw[gray!80, thick] (7,1) -- (-5,-2);

        \filldraw[black] (-1,5) circle (1.5pt) node[anchor=south] {$(0,0,1)$};
        \filldraw[black] (1,-5) circle (1.5pt) node[anchor=north] {$(4,2,3)$};
        \filldraw[black] (-5,-2) circle (1.5pt) node[anchor=east] {$(-5,-4,-6)$};
        \filldraw[black] (7,1) circle (1.5pt) node[anchor=west] {$(0,1,0)$};
        
        \filldraw[black] (0,0) circle (1.5pt) node[anchor=west] {$(2,1,2)$};
        
        \filldraw[black] (-3,-3) circle (1.5pt) node[anchor=north east] {$(-2,-2,-3)$};
        \filldraw[black] (-1,-4) circle (1.5pt) node[anchor=north east] {$(1,0,0)$};
        
        \filldraw[black] (-2,1) circle (1.5pt) node[anchor=north] {$(-1,-1,-1)$};
        
        \filldraw[black] (3,1) circle (1.5pt) node[anchor=north] {$(1,1,1)$};
        
        \end{tikzpicture}
    }
    \resizebox*{0.49\linewidth}{!}{
        \begin{tikzpicture}[scale=0.7]
        
        \draw[black, very thick] (-1,5) -- (1,-5);
        \draw[black, very thick] (-1,5) -- (-5,-2);
        \draw[black, very thick] (-1,5) -- (7,1);
        \draw[black, very thick] (-5,-2) -- (1,-5);
        \draw[black, very thick] (7,1) -- (1,-5);
        \draw[gray!80, thick] (7,1) -- (-5,-2);

        \filldraw[black] (-1,5) circle (1.5pt) node[anchor=south] {\phantom{$(0,0,1)$}};
        \filldraw[black] (1,-5) circle (1.5pt) node[anchor=north] {\phantom{$(4,2,3)$}};
        \filldraw[black] (-5,-2) circle (1.5pt) node[anchor=east] {\phantom{$(-5,-4,-6)$}};
        \filldraw[black] (7,1) circle (1.5pt) node[anchor=west] {\phantom{$(0,1,0)$}};
        
        \filldraw[black] (0,0) circle (1.5pt) node {};
        
        \filldraw[black] (-3,-3) circle (1.5pt) node {};
        \filldraw[black] (-1,-4) circle (1.5pt) node {};
        
        \filldraw[black] (-2,1) circle (1.5pt) node {};
        
        \filldraw[black] (3,1) circle (1.5pt) node {};
        
        \draw[blue, very thick] (-1,5) -- (3,1);
        \draw[blue, very thick] (0,0) -- (3,1);
        \draw[blue, very thick] (1,-5) -- (3,1);
        \draw[blue, very thick] (3,1) -- (7,1);

        \draw[blue, very thick] (0,0) -- (-1,-4);
        \draw[blue, very thick] (0,0) -- (-2,1);
        \draw[blue, very thick] (-2,1) -- (-3,-3);
        \draw[blue, very thick] (-2,1) -- (-1,-4);
        \draw[blue, very thick] (-2,1) -- (-5,-2);
        \draw[blue, very thick] (-2,1) -- (-1,5);

        \draw[blue!60, thick] (7,1) -- (-3,-3);
        \draw[blue!60, thick] (7,1) -- (-1,-4);

        \draw[black, very thick] (-1,5) -- (1,-5);
        \draw[black, very thick] (-1,5) -- (-5,-2);
        \draw[black, very thick] (-1,5) -- (7,1);
        \draw[black, very thick] (-5,-2) -- (1,-5);
        \draw[black, very thick] (7,1) -- (1,-5);

    \end{tikzpicture}
}
\caption{A boundary triangulation of the three-dimensional canonical polytope with ID \href{http://grdb.co.uk/search/toricf3c?ID_cmp=in&ID=547307}{547307} in the Graded Ring Database~\cite{GRDB-toric3}.}
\label{fig:singularity_content} 
\end{figure}
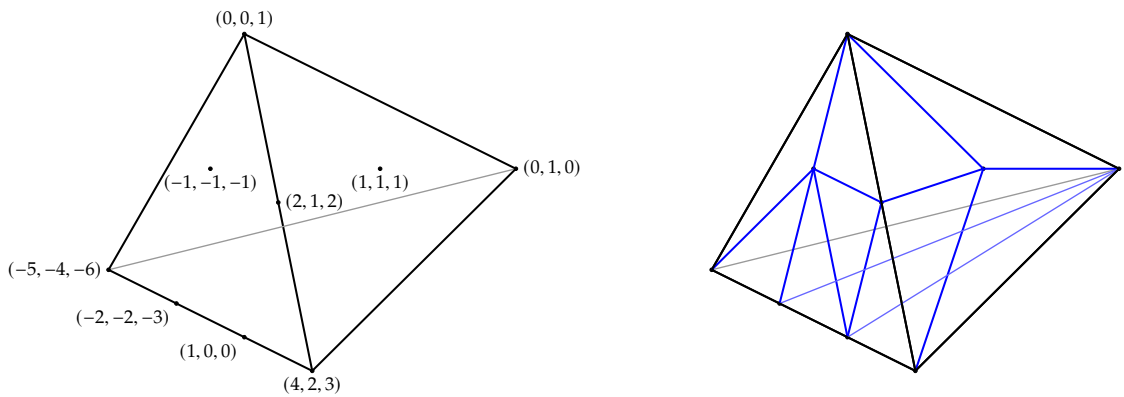

\begin{figure}[ht]
	\centering
    \tdplotsetmaincoords{120}{40}
	\begin{tikzpicture}[line join=bevel,tdplot_main_coords]
        \coordinate (v1) at ( 2,  1,  1);
        \coordinate (v2) at ( 2, -1,  1);
        \coordinate (v3) at ( 1,  1,  1);
        \coordinate (v4) at ( 1, -1,  1);
        \coordinate (v5) at ( 0,  0,  1);
        \coordinate (v6) at (-1,  1, -3);
        \coordinate (v7) at (-1, -1, -3);
        \coordinate (v8) at (-2,  1, -3);
        \coordinate (v9) at (-2, -1, -3);

        \filldraw[color=MidnightBlue!10, ultra thick, fill=MidnightBlue!50, ultra thick] (v6) -- (v7) -- (v9) -- (v8) -- cycle;
        \filldraw[color=MidnightBlue!65, fill=MidnightBlue!65, ultra thick] (1,0,1) -- (v5) -- cycle;
        \filldraw[color=MidnightBlue!10, ultra thick, fill=MidnightBlue!50, ultra thick] (v1) -- (v3) -- (v4) -- (v2) -- cycle;
        
        \filldraw[black] (v1) circle (1.5pt) node[anchor=west] {\tiny $( 2,  1,  1)$};
        \filldraw[black] (v2) circle (1.5pt) node[anchor=south] {\tiny $( 2, -1,  1)$};
        \filldraw[black] (v3) circle (1.5pt) node {};
        \filldraw[black] (v4) circle (1.5pt) node[anchor=east] {\tiny $( 1, -1,  1)$};
        \filldraw[black] (v5) circle (1.5pt) node {};
        \filldraw[black] (v6) circle (1.5pt) node[anchor=west] {\tiny $(-1,  1, -3)$};
        \filldraw[black] (v7) circle (1.5pt) node {};
        \filldraw[black] (v8) circle (1.5pt) node[anchor=north] {\tiny $(-2,  1, -3)$};
        \filldraw[black] (v9) circle (1.5pt) node[anchor=east] {\tiny $(-2, -1, -3)$};

        \filldraw[gray!80] (2, 0, 1) circle (1.5pt) node {};
        \filldraw[gray!80] (1, 0, 1) circle (1.5pt) node {};
        \filldraw[gray!80] (-1, 0, -3) circle (1.5pt) node {};
        \filldraw[gray!80] (-2, 0, -3) circle (1.5pt) node {};
        \filldraw[gray!40] (1, 1, 0) circle (1.5pt) node {};
        \filldraw[gray!40] (0, -1, -1) circle (1.5pt) node {};
        \filldraw[gray!40] (0, 1, -1) circle (1.5pt) node {};
        \filldraw[gray!40] (-1, 0, -2) circle (1.5pt) node {};
        \filldraw[gray!40] (-1, 0, -1) circle (1.5pt) node {};
        \filldraw[gray!40] (0, 0, -1) circle (1.5pt) node {};
        \filldraw[gray!40] (1, 0, 0) circle (1.5pt) node {};
        \filldraw[gray!40] (0, 0, 0) circle (1.5pt) node {};
        \filldraw[gray!40] (1, -1, 0) circle (1.5pt) node {};
        \filldraw[gray!40] (-1, 1, -2) circle (1.5pt) node {};
        \filldraw[gray!40] (-1, -1, -2) circle (1.5pt) node {};

        \draw[black, thick] (v3) -- (v1) -- (v6) -- (v8) -- cycle;
        \draw[black, thin] (v2) -- (v1) -- (v6) -- (v7) -- cycle;
        \draw[black, thick] (v1) -- (v3) -- (v5) -- (v4) -- (v2) -- cycle;
        \draw[black, thin] (v4) -- (v2) -- (v7) -- (v9) -- cycle;
        \draw[black, thin] (v6) -- (v7) -- (v9) -- (v8) -- cycle;
        \draw[black, thick] (v3) -- (v5) -- (v8) -- cycle;
        \draw[black, thick] (v5) -- (v8) -- (v9) -- cycle;
        \draw[black, thick] (v4) -- (v5) -- (v9) -- cycle;

    \end{tikzpicture}
    \caption{A scaffolding of a Laurent polynomial that is mutation-equivalent to~$f_2$.}
    \label{fig:future-scaffolding}
\end{figure}
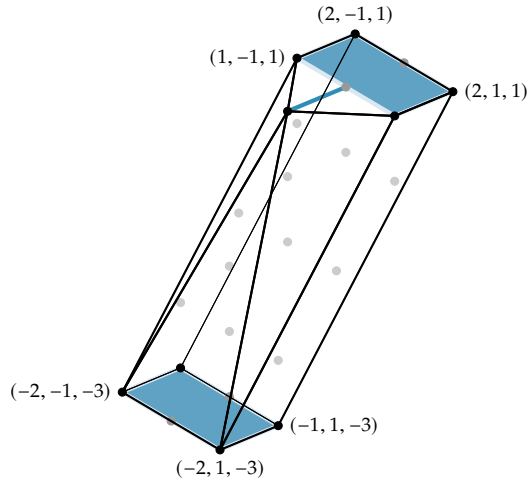

Consider in addition a triangulation of the boundary of~$P$ as shown in Figure~\ref{fig:singularity_content}, and form a fan by taking cones over each triangle in the triangulation. The left-hand front facet gives rise to six smooth cones, and the right-hand front facet gives to four smooth cones. The back facet gives a singularity of type~$\frac{1}{5}(1,1,4)$, and the bottom facet gives three singularities of type~$\frac{1}{3}(1,1,2)$; these singularities are terminal and qG\nobreakdash-rigid. This suggests a basket
\[
    \cB = \Big\{ \textstyle 3 \times \frac{1}{3}(1,1,2), \frac{1}{5}(1,1,4) \Big\}
\]
which agrees with the basket calculated from the Ehrhart series of~$P^*$, that is, with the basket of the possible \QQFano{} Hilbert series with ID~\href{http://grdb.co.uk/search/fano3?id_cmp=in&id=29915}{29915} in the Graded Ring Database\footnote{Recall that we expect that, if a Fano variety~$X$ corresponds under mirror symmetry to a Laurent polynomial~$f$, then there is a qG\nobreakdash-degeneration from~$X$ to~$X_f$. This implies that the Hilbert series of~$X$ coincides with the Hilbert series of~$X_f$, and hence with the Ehrhart series of the dual polytope~$P^*$ where~$P = \Newt(f)$. Thus the Hilbert series of~$X$ is determined by the Newton polytope~$P$ of~$f$. Furthermore the singularities of a \QQFano{} threefold~$X$ are uniquely determined by its Hilbert series, and hence by~$P$. The fact that the singularities of a \QQFano{} threefold are determined by its Hilbert series seems to be a combinatorial accident: it is established by looking case by case through the Graded Ring Database~\cite{BrownKasprzyk2022}, and lacks a geometric proof. We see no reason for the corresponding statement to be true for higher-dimension \QQFano{} varieties. Thus it is likely that, in dimension four and higher, even the basket~$\cB$ of a \QQFano{} qG\nobreakdash-generization~$X$ of~$X_f$ will depend on the rigid MMLP~$f$ that corresponds to~$X$, and not just on the Newton polytope of~$f$.}. In this particular example any triangulation such that the faces are divided into empty triangles will give the same result: the front two faces are both at lattice height~1 above the origin, and so any triangulation of them will give rise to a total of~$10$ smooth cones. There is no choice for the triangulation of the other two faces, at least if we insist on the triangles on the bottom face being empty.

For a second example, consider the Laurent polynomial
\[
    f = \frac{(x^2 + 2x + yz^3 + 3yz^2 + 3yz + y + 1)^2}{xyz} - 12
\]
which is mirror to the smooth Fano $3$\nobreakdash-fold~$V_6$. (It is easy to check that~$f$ is mutation-equivalent to the Laurent polynomial mirror to~$V_6$ given in~\cite[Appendix~A]{CoatesCortiGalkinKasprzyk2016}.) The Newton polytope~$Q$ of~$f$ is pictured in Figure~\ref{fig:V6_mirror}. Since the Fano variety~$V_6$ is smooth, we expect that~$Q$ should have empty basket; indeed the Ehrhart series of~$Q^*$, which is the possible \QQFano{} Hilbert series with ID~\href{http://grdb.co.uk/search/fano3?id_cmp=in&id=24076}{24076} in the Graded Ring Database, has empty basket.
\begin{figure}[ht]
	\centering
    \tdplotsetmaincoords{120}{40}
	\begin{tikzpicture}[line join=bevel,tdplot_main_coords]
        \coordinate (v1) at (3,-1,-1);
        \coordinate (v2) at (-1,1,5);
        \coordinate (v3) at (-1,1,-1);
        \coordinate (v4) at (-1,-1,-1);

        \filldraw[black] (v1) circle (1.5pt) node[anchor=west] {$(3,-1,-1)$};
        \filldraw[black] (v2) circle (1.5pt) node[anchor=south] {$(-1,1,5)$};
        \filldraw[black] (v3) circle (1.5pt) node[anchor=north] {$(-1,1,-1)$};
        \filldraw[black] (v4) circle (1.5pt) node[anchor=east] {$(-1,-1,-1)$};
        
        
        \filldraw[gray!80] (-1, 0, -1) circle (1.5pt) node {};
        \filldraw[gray!80] (-1, 0, 0) circle (1.5pt) node {};
        \filldraw[gray!80] (-1, 0, 1) circle (1.5pt) node {};
        \filldraw[gray!80] (-1, 0, 2) circle (1.5pt) node {};
        \filldraw[gray!80] (-1, 1, 0) circle (1.5pt) node {};
        \filldraw[gray!80] (-1, 1, 1) circle (1.5pt) node {};
        \filldraw[gray!80] (-1, 1, 2) circle (1.5pt) node {};
        \filldraw[gray!80] (-1, 1, 3) circle (1.5pt) node {};
        \filldraw[gray!80] (-1, 1, 4) circle (1.5pt) node {};
        \filldraw[gray!80] (1, 0, -1) circle (1.5pt) node {};
        \filldraw[gray!80] (1, 0, 0) circle (1.5pt) node {};
        \filldraw[gray!80] (1, 0, 1) circle (1.5pt) node {};
        \filldraw[gray!80] (1, 0, 2) circle (1.5pt) node {};
        \filldraw[gray!40] (0, -1, -1) circle (1.5pt) node {};
        \filldraw[gray!40] (0, 0, 2) circle (1.5pt) node {};
        \filldraw[gray!40] (0, 0, 1) circle (1.5pt) node {};
        \filldraw[gray!40] (2, -1, -1) circle (1.5pt) node {};
        \filldraw[gray!40] (0, 0, 0) circle (1.5pt) node {};
        \filldraw[gray!40] (0, 0, -1) circle (1.5pt) node {};
        \filldraw[gray!40] (1, -1, -1) circle (1.5pt) node {};

        \draw[gray!80, thick] (v1) -- (v2) -- (v4) -- (v1);
        \draw[black, very thick] (v1) -- (v2) -- (v3) -- (v1);
        \draw[black, very thick] (v2) -- (v3) -- (v4) -- (v2);

    \end{tikzpicture}
    \caption{The Newton polytope~$Q$ of a Laurent polynomial mirror to~$V_6$.}
    \label{fig:V6_mirror}
\end{figure}
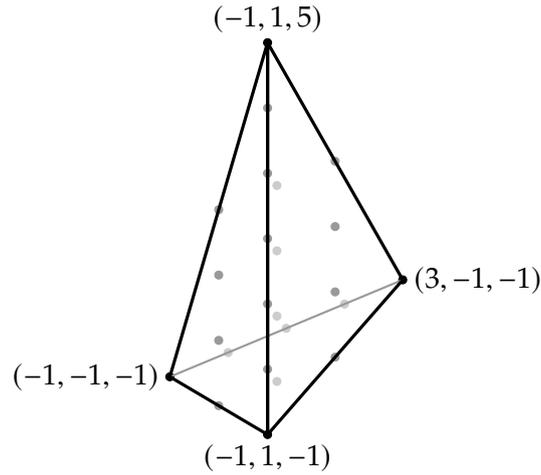
As before, let us consider a triangulation of the boundary of~$Q$, and form a fan by taking cones over each triangle in the triangulation. Each face of~$Q$ other than the back face has lattice height~$1$ above the origin, and the back face has lattice~height~$2$. So let us insist that our triangulation contains only empty triangles on the front faces and bottom face; these empty triangles will give rise to smooth cones in the fan. 
\begin{figure}[ht]
	\centering
	\begin{tikzpicture}[line join=bevel, scale=0.7]
        \coordinate (v1) at (-2,0);
        \coordinate (v2) at (0,4);
        \coordinate (v3) at (2,0);

        \filldraw[black] (v1) circle (1.5pt) node[anchor=north east] {$(-1,-1,-1)$};
        \filldraw[black] (v2) circle (1.5pt) node[anchor=south] {$(-1,1,5)$};
        \filldraw[black] (v3) circle (1.5pt) node[anchor=north west] {$(3,-1,-1)$};

        \coordinate (p1) at (-1,2);
        \coordinate (p2) at (0,2);
        \coordinate (p3) at (1,2);
        \coordinate (q1) at (-1,0);
        \coordinate (q2) at (0,0);
        \coordinate (q3) at (1,0);

        \filldraw[gray!80] (p1) circle (1.5pt) node {};
        \filldraw[gray!80] (p2) circle (1.5pt) node {};
        \filldraw[gray!80] (p3) circle (1.5pt) node {};
        \filldraw[gray!80] (q1) circle (1.5pt) node {};
        \filldraw[gray!80] (q2) circle (1.5pt) node {};
        \filldraw[gray!80] (q3) circle (1.5pt) node {};

        \draw[black, very thick] (v1) -- (v2) -- (v3) -- (v1);
    \end{tikzpicture}
    \begin{tikzpicture}[line join=bevel, scale=0.7]
        \coordinate (v1) at (-2,0);
        \coordinate (v2) at (0,4);
        \coordinate (v3) at (2,0);

        \filldraw[black] (v1) circle (1.5pt) node[anchor=north east] {\phantom{$(-1,-1,-1)$}};
        \filldraw[black] (v2) circle (1.5pt) node[anchor=south] {\phantom{$(-1,1,5)$}};
        \filldraw[black] (v3) circle (1.5pt) node[anchor=north west] {\phantom{$(3,-1,-1)$}};

        \coordinate (p1) at (-1,2);
        \coordinate (p2) at (0,2);
        \coordinate (p3) at (1,2);
        \coordinate (q1) at (-1,0);
        \coordinate (q2) at (0,0);
        \coordinate (q3) at (1,0);

        \filldraw[gray!80] (p1) circle (1.5pt) node {};
        \filldraw[gray!80] (p2) circle (1.5pt) node {};
        \filldraw[gray!80] (p3) circle (1.5pt) node {};
        \filldraw[gray!80] (q1) circle (1.5pt) node {};
        \filldraw[gray!80] (q2) circle (1.5pt) node {};
        \filldraw[gray!80] (q3) circle (1.5pt) node {};

        \draw[blue, thick] (v2) -- (q2) -- (v1) -- cycle;
        \draw[blue, thick] (v2) -- (q2) -- (v3) -- cycle;

        \draw[black, very thick] (v1) -- (v2) -- (v3) -- (v1);
    \end{tikzpicture}
    \caption{A triangulation of the back facet of~$Q$.}
    \label{fig:V6_mirror_facet}
\end{figure}
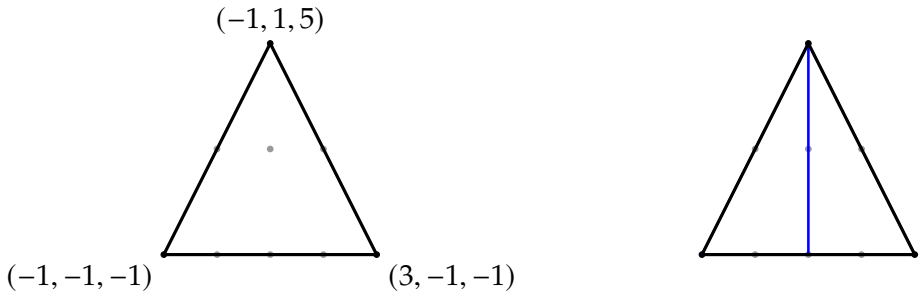
For the back face, let us take the triangulation shown in Figure~\ref{fig:V6_mirror_facet}. Note that the height of the facet is~$2$, and that both triangles shown are~$2$\nobreakdash-fold dilations of a standard two-dimensional lattice simplex; note also that all vertices in the triangulation are primitive. This suggests, by analogy with the construction of singularity content in two dimensions, that the singularities that correspond to these cones should be qG\nobreakdash-smoothable -- that is, we expect these cones to give rise to three-dimensional~$T$-singularities. This would give empty basket for~$Q$, in agreement with the discussion above.

One challenge in passing from this discussion to a satisfactory definition of singularity content in three dimensions is that in general there are many different triangulations of the boundary of a three-dimensional polytope~$Q = \Newt{f}$, and it is not clear (at least to us) which such triangulations should be admissible. We expect that the admissible triangulations will be determined by set~$S_f$ of mutations that~$f$ permits; this was defined just before Proposition~\ref{pro:rigid}. In particular the admissible triangulations will depend on~$f$ and not just on the underlying polytope~$Q$. One can see hints of this phenomenon in the work of Corti--Hacking--Petracci on smoothing Gorenstein toric Fano varieties: see~\cite{CoatesCortiDaSilva2019}. It is likely that singularity content in higher dimensions will closely resemble the Corti--Filip--Petracci notion of zero-mutable Laurent polynomials~\cite{CortiFilipPetracci2020}.
\begin{description}
    \item[Problem E] Give a combinatorial definition of singularity content in all dimensions.
\end{description}
Whatever the definition is, in a fully satisfactory theory there should be a bijection between:
\begin{itemize}
    \item the data defining the singularity content of rigid MMLPs~$f$ with Newton polytope~$P$;
    \item the set of mutation-equivalence classes of such~$f$;
    \item the smoothing components of the qG\nobreakdash-deformation space of~$X_P$.
\end{itemize}
This is the picture that holds in two dimensions~\cite{AkhtarCoatesCortiHeubergerKasprzykOnetoPetracciPrinceTveiten2016,CoatesKasprzykPittonTveiten2021}, and that we expect to generalise.

\subsection*{Acknowledgments}\label{sec:acknowledgements}
The research presented here has been guided in a fundamental way by ideas of and joint work with Alessio Corti. It forms part of a broader program, initiated by Corti and Vasily Golyshev, to discover and understand the classification of Fano varieties using mirror symmetry. We thank Corti, Sergey Galkin, Golyshev, Andrea Petracci, and Thomas Prince for many useful conversations. In particular, Conjecture~\ref{conj:mirror} arose from conversations with Corti and Golyshev at the workshop \emph{Motivic Structures on Quantum Cohomology and Pencils of CY Motives} at the Max Planck Institute for Mathematics, Bonn in September 2014.

The computations that underlie this work were performed using the Imperial College High Performance Computing Service and the compute cluster at the Department of Mathematics, Imperial College London. We thank Andy Thomas and Matt Harvey for invaluable technical assistance. TC~is supported by ERC Consolidator Grant~682603 and EPSRC Programme Grant~EP/N03189X/1. LH~is supported by Leverhulme grant RPG-2021-149 and Projet \'Etoiles montantes R\'egion Pays de la Loire. AK~is supported by EPSRC Fellowship~EP/N022513/1.
\bibliographystyle{plain}

\end{document}